\documentclass[reqno,a4paper]{amsart}
\usepackage{eucal,amsfonts,amssymb,amsmath,amsthm,epsfig,mathrsfs}

\usepackage{amscd,amsxtra}
\usepackage{enumerate}
\usepackage{latexsym}

\allowdisplaybreaks

 \makeatletter \@addtoreset{equation}{section}

\makeatother \makeatletter

\newtheorem{thm}{Theorem}[section]
\newtheorem{hyp}[thm]{Hypotheses}{\rm}
{\rm}
\newtheorem{lemm}[thm]{Lemma}

\newtheorem{prop}[thm]{Proposition}
\newtheorem{defi}[thm]{Definition}
\newtheorem{rmk}[thm]{Remark}{\rm}

\newcommand{\R}{{\mathbb R}}
\newcommand{\N}{{\mathbb N}}

\newcommand{\Rd}{\mathbb R^d}

\newcommand{\supp}{{\rm{supp}}\,}

\newcommand{\bd}{\begin{defi}}
\newcommand{\ed}{\end{defi}}
\newcommand{\nnm}{\nonumber}
\newcommand{\be}{\begin{equation}}
\newcommand{\ee}{\end{equation}}
\newcommand{\barr}{\begin{array}}
\newcommand{\earr}{\end{array}}
\newcommand{\bmn}{\begin{eqnarray}}
\newcommand{\emn}{\end{eqnarray}}
\newcommand{\bnm}{\begin{eqnarray*}}
\newcommand{\enm}{\end{eqnarray*}}
\newcommand{\bln}{\begin{subequations}}
\newcommand{\eln}{\end{subequations}}

\newcommand{\ba}{\begin{align}}
\newcommand{\ea}{\end{align}}
\newcommand{\banm}{\begin{align*}}
\newcommand{\eanm}{\end{align*}}

\title[Summability improving in nonautonomous Kolmogorov equations]{On improvement of summability properties in nonautonomous Kolmogorov equations}
\author[L. Angiuli, L. Lorenzi]{Luciana Angiuli and Luca Lorenzi}
\address{Dipartimento di Matematica, Universit\`a degli Studi di Parma, Parco Area delle Scienze 53/A, I-43124 Parma, Italy.}
\email{(luciana.angiuli, luca.lorenzi)@unipr.it}

\date{}
\begin{document}

\begin{abstract}
Under suitable conditions, we obtain some characterization of supercontractivity,
ultraboundedness and ultracontractivity of the evolution operator $G(t,s)$ associated to a class of
nonautonomous second order parabolic equations with unbounded coefficients
defined in $I\times\R^d$, where $I$ is a right-halfline. For this purpose, we establish
an Harnack type estimate for $G(t,s)$ and a family of logarithmic
Sobolev inequalities with respect to the unique tight evolution system of measures
$\{\mu_t: t \in I\}$ associated to $G(t,s)$. Sufficient conditions for the supercontractivity, ultraboundedness and
ultracontractivity to hold are also provided.
\end{abstract}

\maketitle
\section{Introduction}
Let ${\mathcal A}$ be an autonomous second order uniformly elliptic operator
with unbounded coefficients defined in $\Rd$. It is well known that, under suitable assumptions
on its coefficients, a Markov semigroup $T(t)$ can be associated in $C_b(\Rd)$ to the operator $\mathcal{A}$.
More precisely, for any $f\in C_b(\Rd)$, $T(t)f$ is the value at $t$ of the (unique) bounded classical solution of the Cauchy problem
\begin{eqnarray*}
\left\{
\begin{array}{ll}
D_tu(t,x)={\mathcal{A}}u(t,x),\quad\quad & (t,x)\in (0,+\infty)\times \Rd,\\[1mm]
u(0,x)= f(x),\quad\quad & x\in \Rd.
\end{array}\right.
\end{eqnarray*}
Under somehow stronger assumptions on the coefficients of the operator $\mathcal{A}$, an invariant measure $\mu$ can be
associated to the semigroup $T(t)$ which can be extended to a contractive semigroup in $L^p(\Rd,\mu)$ for any $p\in [1,+\infty)$.

It is also well known that in some cases $T(t)$ improves summability, i.e., it maps $L^p(\Rd,\mu)$ into $L^q(\Rd,\mu)$ for some $q>p$ and $t> \overline{t}(p,q)\geq 0$, and
\begin{equation}\label{est_intro}
C_{p,q}(t):=\|T(t)\|_{\mathcal{L}(L^p(\Rd,\mu),L^q(\Rd,\mu))}<+\infty.
\end{equation}
This property is called \emph{hypercontractivity} if $p,q \in(1,+\infty)$, $\overline{t}(p,q)>0$
and $C_{p,q}(t)=1$, \emph{supercontractivity} if $p,q \in(1,+\infty)$
and $\overline{t}(p,q)=0$, \emph{ultraboundedness} if $p \in(1,+\infty)$, $q=+\infty$
and $\overline{t}(p,q)=0$.
If $p\in [1,+\infty)$ this last property is called \emph{ultracontractivity}.

Estimate \eqref{est_intro}
is equivalent to the occurrence of some functional inequalities satisfied by the
invariant measure $\mu$. We refer to \cite{Gro75Log}, the pioneering work on such topics,
where a characterization of the hypercontractivity
and the supercontractivity of the semigroup $T(t)$ is given in terms of some logarithmic Sobolev inequalities.

Ultraboundedness and ultracontractivity have been widely studied in the autonomous setting,
mainly in the symmetric case (where they are equivalent).
 The first result in this direction is due to Davies and Simon \cite{DavSim84Ult,Davies} that, following the idea of Gross and
requiring some additional integrability conditions, connect ultracontractivity with a family of logarithmic Sobolev inequalities.

Other different approaches to study ultracontractivity have been also suggested by \cite{CarKusStr87Upp} and, more recently, by \cite{wan00Fun}.

On the other hand, to the best of our knowledge, results on summability improving have been not yet studied in the nonautonomous case.

In the recent paper \cite{AngLorLun} we have dealt with
hypercontractivity and we have extended the connection with logarithmic
Sobolev inequalities in a nonautonomous setting, where the semigroup $T(t)$ and the invariant measure $\mu$ are replaced, respectively,
by a Markov evolution operator $G(t,s)$
and an evolution system of measures $\{\mu_t\}$.

In this paper we are interested in exploiting some regularizing properties, stronger
than hypercontractivity, for the evolution operator $G(t,s)$,
and in characterizing them in terms of suitable
inequalities satisfied by an evolution system of measures $\{\mu_t\}$.

Let $I$ be an open right halfline and for every $t \in I$ consider
the nonautonomous second order differential operator $\mathcal{A}(t)$
defined on smooth functions $\zeta$ by
\begin{eqnarray*}
(\mathcal{A}(t)\zeta)(x)=\textrm{Tr}(Q(t)D^2\zeta(x))+ \langle b(t,x),
\nabla \zeta(x)\rangle,\qquad\;\, x\in\R^d.
\end{eqnarray*}
We assume some smoothness on  $Q=[q_{ij}]_{i, j=1, \ldots, d}$ and
$b = (b_1, \ldots, b_d)$, defined in $I$ and $I\times\Rd$, respectively. Moreover, we require that the coefficients $q_{ij}$ are bounded and that
the operators $\mathcal{A}(t)$ are uniformly elliptic, i.e., there exists
a positive constant $\eta_0$ such that
\begin{eqnarray*}
\langle Q(t)\xi,\xi\rangle\geq\eta_0|\xi|^2,\qquad\;\, t\in I,\;\,\xi\in \R^d.
\end{eqnarray*}
Assuming the existence of a Lyapunov function, for every $s\in I$ and $f\in C_b(\Rd)$, the nonautonomous Cauchy problem
\begin{eqnarray*}
\left\{
\begin{array}{ll}
D_tu(t,x)={\mathcal{A}}(t)u(t,x),\quad\quad & (t,x)\in (s,+\infty)\times \Rd,\\[1mm]
u(s,x)= f(x),\quad\quad & x\in \Rd,
\end{array}\right.
\end{eqnarray*}
admits a unique bounded classical solution $u=G(\cdot,s)f$, where $G(t,s)$ is a Markov evolution operator. The function
$G(\cdot,s)f$ belongs to $C^{1+\alpha/2,2+\alpha}_{\rm loc}((s,+\infty)\times \R^d)$ and admits the following
representation formula
\begin{equation}\label{rep_nucleo_intro}
(G(t,s)f)(x)=\int_{\R^d}g_{t,s}(x,y)f(y)dy,\qquad\;\,s<t,\;\,x\in\R^d,\;\, f\in C_b(\Rd),
\end{equation}
where $g_{t,s}:\R^d\times\R^d\to\R$ is a positive function such that
 $\|g_{t,s}(x,\cdot)\|_{L^1(\Rd)}=1$ for any $t,s\in I$, with $t>s$, and any $x\in\Rd$.

The existence of a Lyapunov function such that
\begin{equation*}
\lim_{|x|\to +\infty}\varphi(x)=+\infty \quad\textrm{and}\quad
(\mathcal{A}(t)\varphi)(x)\leq a-\gamma\,\varphi(x),  \quad (t,x)\in I\times \Rd,
\end{equation*}
for some positive constants $a$ and $\gamma$, allows (see \cite{KunLorLun09Non}) to prove the existence of tight evolution
 systems of measures $\{\mu_t: t\in I\}$, i.e., families of Borel probability measures such that $\mu_t(B(0,R))$ tends to $1$
 as $R\to +\infty$, uniformly with respect to $t\in I$, and
\begin{equation}\label{invariance_intro}
\int_{\Rd}(G(t,s)f)(y)d\mu_t(y)=\int_{\Rd}f(y)d\mu_s(y), \qquad\;\, t>s\in I, \;\, f\in C_b(\Rd).
\end{equation}
The interest in evolution systems of measures is due to the good properties that the
evolution operators enjoy in the $L^p$-spaces related to these systems.
Indeed, using \eqref{invariance_intro} and the density of $C^{\infty}_c(\Rd)$ in $L^p(\Rd,\mu_t)$
for every $t \in I$, the evolution operator can be extended to a contraction (still denoted by
$G(t,s)$) from $L^p(\Rd,\mu_s)$ to $L^p(\Rd,\mu_t)$ for every $p\in [1,+\infty)$.

In this context a generalization to the nonautonomous case  of the definitions of hypercontractivity,
supercontractivity, ultracontractivity and ultraboundedness (see Definition \ref{definitions})
and of their
characterizations is significant and interesting.

As it has  been already remarked, in \cite{AngLorLun} hypercontractivity
of the evolution operator $G(t,s)$ has been studied,
assuming some stronger assumption than the minimal ones that guarantee the basic properties of $G(t,s)$ and the existence of an evolution system of measures $\{\mu_t: t\in I\}$.
In fact, if the dissipativity condition
\begin{eqnarray}\label{dissipativity_intro}
\langle \nabla_xb(t,x)\xi, \xi \rangle \leq r_0|\xi|^2, \qquad\;\, t\in I, \;\,x,\xi \in \R^d
\end{eqnarray}
is satisfied for some $r_0<0$, then the logarithmic Sobolev inequality (in short LSI) for the unique tight evolution system of measures $\{\mu_s: s\in I\}$
\begin{equation}
\int_{\Rd} f^2\log \left(\frac{|f|}{\|f\|_{L^2(\Rd,\mu_s)}}\right)\,d\mu_s(x) \le C\int_{\Rd}|\nabla f|^2d\mu_s(x) ,
\label{LSI}
\tag{LSI}
\end{equation}
holds for any $s\in I$, $f\in H^1(\Rd,\mu_s)$ and some positive constant $C$, independent of $f$ and $s$. The hypercontractivity of $G(t,s)$
in $L^p$ spaces related to the unique tight evolution system of measures, is obtained as a
consequence of the (LSI).

In general, evolution systems of measures are infinitely many (see e.g., \cite{GeisLun08}). Among all of them, the unique tight system as a prominent role.
Indeed, it is related to the asymptotic behaviour of $G(t,s)$ as $t\to +\infty$. As it has been proved in \cite{AngLorLun}, under condition \eqref{dissipativity_intro}
\begin{eqnarray*}
\lim_{t\to +\infty}\int_{\Rd}|G(t,s)f-m_s(f)|^pd\mu_t(x) =0,
\end{eqnarray*}
uniformly with respect to $f\in L^p(\Rd,\mu_s)$, $p\in [1,+\infty)$, where $m_s(f)$ denotes the average of $f$ with respect to the measure $\mu_s$.

In this paper, we assume that condition \eqref{dissipativity_intro} holds true and
consider the unique tight evolution system of measures $\{\mu_s: s\in I\}$.

We first prove that the supercontractivity property of the evolution operator
$G(t,s)$ is equivalent to the validity of the following family of logarithmic
Sobolev inequalities (in short LSI$_\varepsilon$)
\begin{equation}\label{fam_log_sob}
\int_{\Rd} f^2\log \left(\frac{|f|}{\|f\|_{L^2(\Rd,\mu_s)}}\right)
d\mu_s(x) \le \varepsilon \|\,|\nabla f|\,\|_{L^2(\Rd,\mu_s)}^2 +\beta(\varepsilon)\|f\|_{L^2(\Rd,\mu_s)}^2,
\end{equation}
for every $s\in I$, $f\in H^1(\Rd,\mu_s)$, $\varepsilon >0$ and some
positive decreasing function $\beta$.
We follow the method of \cite{RocWan03Sup} that, on a Riemann manifold $M$,
deals with the diffusion semigroup $P_t$  generated by
the autonomous operator $L=\Delta+Z\nabla$ with Neumann boundary conditions on $\partial M$,
where $Z$ is a $C^1$-vector field satisfying a curvature condition.
The condition on the curvature is used to deduce the following logarithmic Sobolev inequality satisfied by $P_t$
\begin{equation}\label{LSI_Pt}
P_t(f^2\log f^2)\leq \frac{2(e^{2Kt}-1)}{K}P_t|\nabla f|^2+(P_t f^2)\log(P_t f^2),
\end{equation}
which holds for every $f\in C_0^{\infty}(M)$, $t>0$ and some positive constant $K>0$.

The starting point of our analysis is the analogue of \eqref{LSI_Pt} in the nonautonomous case;
 we prove a logarithmic Sobolev inequality satisfied by the probability measures
$g_{t,s}(x, dy)=g_{t,s}(x,y)dy$ defined in \eqref{rep_nucleo_intro}. More precisely, we show that
\begin{equation}\label{log_sob_G_2_intro}
G(t,s)(f^2 \log f^2)\leq   \frac{4 \Lambda}{|r_0|} (1-e^{2r_0(t-s)})
G(t,s)(|\nabla f|^2)+(G(t,s)f^2) \log (G(t,s)f^2),
\end{equation}
for every $f\in C^1_b(\Rd)$ and $t,\,s \in I$ such that $t\ge s$.
The key tool for the proof of estimate \eqref{log_sob_G_2_intro}
(and of many results in the paper) is the pointwise gradient estimate
\begin{equation}
\label{grad_est_intro}
|(\nabla_x G(t,s)f)(x)|\leq e^{r_0(t-s)}(G(t,s)|\nabla f|)(x),\qquad\;\,t>s,\;\,x\in\Rd,\;\,f\in C^1_b(\Rd),
\end{equation}
that has been proved in \cite{KunLorLun09Non} under the assumption \eqref{dissipativity_intro}
(which is equivalent to the condition considered in \cite{RocWan03Sup}).
Even in the autonomous case, \eqref{grad_est_intro} does not hold when the diffusion coefficients
depend on $x$ and they do not satisfy the condition in \cite{wang}. This is the reason why
we confine ourself to the case of diffusion coefficients depending only on $t$.

Another important consequence of \eqref{grad_est_intro} is the Harnack type estimate
\begin{equation}\label{Harnack_intro}
|(G(t,s)f)(x)|^2\leq (G(t,s)|f|^2)(y)\exp{\left(\frac{|x-y|^2}{2\eta_0(t-s)}\right)},\qquad\;\,t>s,\;\,x,y\in\Rd,
\end{equation}
satisfied by any $f\in C_b(\Rd)$.
Estimate \eqref{Harnack_intro} and LSI$_\varepsilon$ allow us to prove a second criterion for supercontractivity:
we show that the integrability with respect to the measures $\{\mu_t:t\in I\}$ (uniform in $t$) of the Gaussian functions
$\varphi_{\lambda}(x):=e^{\lambda |x|^2}$, for every $\lambda>0$, is another condition equivalent to the supercontractivity of $G(t,s)$.
This second characterization is useful in order to provide a sufficient condition for the evolution operator $G(t,s)$ to be supercontractive
as stated in Theorem \ref{suff_cond_super}.

The Harnack type estimate \eqref{Harnack_intro} is also the key tool to prove that,
if $G(t,s)\varphi_\lambda\in L^\infty(\Rd)$ for every $t>s\in I$ and
$\lambda>0$, and
\begin{equation}
\sup_{{s,t\in I}\atop{t-s \ge \delta}}\|G(t,s)\varphi_{\lambda}\|_{\infty}<+\infty,\qquad\;\,\delta,\lambda>0,
\label{unif-limit-intr}
\end{equation}
then $G(t,s)$ is ultrabounded.
We provide a sufficient condition for
$G(t,s)\varphi_\lambda$ to be bounded for every $t>s\in I$ and every $\lambda>0$ (see Theorem \ref{ultrathm}).

Actually, condition \eqref{unif-limit-intr} is also necessary to get ultraboundedness.
We prove the necessity of this condition using the characterization of the supercontractivity property
in terms of the family of inequalities \eqref{fam_log_sob}.

A quite sharp condition to get ultraboundedness of $G(t,s)$ is given in terms of the inner product between the drift $b(t,x)$ and $x$, which has to satisfy
\begin{equation}\label{cond_ubb_intro}
\langle b(t,x),x\rangle \leq -K_1|x|^2(\log|x|)^{\alpha},\qquad\:\,t\in I,|x|\ge R,
\end{equation}
for some positive constants $K_1$, $\alpha>1$ and $R>1$.

Under some stronger condition than \eqref{cond_ubb_intro} on $\langle b(t,x),x\rangle$, we prove that $G(t,s)$ is bounded from $L^1(\Rd,\mu_s)$ to $L^2(\Rd,\mu_t)$, hence it is ultracontractive.

Then we extend supercontractivity, ultraboundedness and ultracontractivity to evolution operators associated to nonautonomous operators with non zero potential term.

Finally, we establish some consequences of the regularizing properties of $G(t,s)$.
More precisely, we get an $L^{\infty}$-estimate for the integral kernel $g_{t,s}$ of $G(t,s)$ (see \eqref{rep_nucleo_intro})
and some $L^2$-uniform integrability properties of $G(t,s)$.

The paper is organized as follows. First, in Section \ref{preliminary}, we state our main assumptions,
we collect some known results on the evolution operator $G(t,s)$ and we give the definition of supercontractivity, ultraboundedness and ultracontractivity in our nonautonomous setting.
Section \ref{section_super} is devoted to prove two criteria
for the supercontractivity property of $G(t,s)$.
In Section \ref{section_ultrabdd} we provide a characterization of ultraboundedness for
$G(t,s)$ in terms of the boundedness of the function $G(t,s)\varphi_\lambda$.
Section \ref{section_ultracontra} concerns the $L^1$-$L^2$ boundedness of $G(t,s)$ and the consequent ultracontractivity property.
Finally, in Section \ref{sect-6}, we collect some consequences of the ultracontractivity of $G(t,s)$.

\subsection*{Notations}
Let $k\in \N \cup \{0,+ \infty\}$, we consider the usual space $C^k(\Rd)$, as well as $C^k_b(\Rd)$,
the subspace of $C^k(\Rd)$ consisting of
bounded functions with bounded derivatives up to the $k$-th order.
We use the subscript ``$c$'' instead of ``$b$''  for the subsets of the
above spaces consisting of functions  with compact support.

If $J \subset \R$ is an interval and $\alpha\in (0,1)$, $C^{\alpha/2,\alpha}(J \times\Rd)$
denotes the usual parabolic H\"older space. We use the subscript ``loc'' to denote
the space of all $f\in C(J\times \Rd)$ which are $(\alpha/2,\alpha)$-H\"older continuous in any compact set of $J\times \Rd$.

Let $\mu$ be a probability measure on $\Rd$ and $1\leq p<\infty$. We denote by $L^p(\Rd, \mu)$ the
set of $\mu$-measurable functions $f:\Rd\to \R\cup\{\pm \infty\}$ such that $\|f\|_{p,\mu}^p:=\int_{\Rd}|f|^pd\mu(x)<+\infty$. When $d\mu=dx$ is the Lebesgue measure, we
simply write $\|f\|_p$.
If $p=+\infty$ then $L^{\infty}(\Rd, \mu)=L^\infty(\Rd)$ is endowed with the sup-norm $\|\cdot\|_{\infty}$.
The space $H^1(\Rd,\mu)$ consists of all the functions which belong to $L^2(\Rd, \mu)$ together with their first order distributional derivatives.

Let $T$ be an operator mapping $L^p(\Rd,\mu)$ to $L^q(\Rd,\nu)$  for
$1\leq p\leq q\le+\infty$ where $\mu,\nu$ are two probability measures on $\Rd$.
If no confusion may arise, we denote by $\|T\|_{p\to q}$ the operator norm
$\|T\|_{\mathcal{L}(L^p(\Rd,\mu),L^q(\Rd,\nu))}$.

About partial derivatives, the notations $D_tf:=\frac{\partial f}{\partial t}$,
$D_if:=\frac{\partial f}{\partial x_i}$, $D_{ij}f:=\frac{\partial^2f}{\partial x_i\partial x_j}$ are extensively used.

About matrices and vectors, we denote by $\textrm{Tr}(Q)$ and $\langle x,y\rangle$ the trace of the square matrix
$Q$ and the  inner product
of the vectors $x,y\in\Rd$, respectively.

The  ball in $\R^d$ centered at $0$ with  radius $r>0$ is denoted by $B(0,r)$.

Finally, we set $0 \log 0=0$ by definition.

\section{Assumptions, definitions and a review of some properties of $G(t,s)$}\label{preliminary}
Let $I$ be an open right  halfline. For $t\in I$ we consider
linear second order differential
operators $ \mathcal{A}(t) $
defined on smooth functions $\zeta$ by
\begin{align*}
(\mathcal{A}(t)\zeta)(x)&=\sum_{i,j=1}^d q_{ij}(t)D_{ij}\zeta(x)+
\sum_{i=1}^d b_i(t,x)D_i\zeta(x)\\
&= \textrm{Tr}(Q(t)D^2\zeta(x))+ \langle b(t,x), \nabla \zeta(x)\rangle,\qquad\;\, x\in\R^d,
\end{align*}
under the following assumptions on their coefficients.

\begin{hyp}\label{hyp1}
\begin{enumerate}[\rm (i)]
\item
$q_{ij}\in C^{\alpha/2}_{\rm loc}(I)$ and
$b_{i} \in C^{\alpha/2,\alpha}_{\rm loc}(I\times \R^d)$ $(i,j=1,\dots,d)$ for
some $\alpha \in (0,1)$;
\item
the matrix $Q(t)=[q_{ij}(t)]_{i,j=1, \ldots, d}$
is symmetric for every $t\in I$ and there exist $0<\eta_0<\Lambda$ such that
\begin{equation}
\label{ell}
\eta_0|\xi|^2\le \langle Q(t)\xi,\xi\rangle\leq \Lambda|\xi|^2,\qquad\;\, (t,\xi )\in I\times \Rd;
\end{equation}
\item
there exists $\varphi\in C^2(\R^d)$ with positive values such that
\begin{equation}
\label{Lyapunov}
\;\;\;\;\;\qquad\lim_{|x|\to +\infty}\varphi(x)=+\infty \quad\textrm{and}\quad
(\mathcal{A}(t)\varphi)(x)\leq a-\gamma\,\varphi(x),  \quad (t,x)\in I\times \Rd,
\end{equation}
for some positive constants $a$ and $\gamma$;
\item
the first order spatial derivatives of $b_i$ exist, belong to
$C^{\alpha/2,\alpha}_{\rm loc}(I\times \Rd)$ for any $i=1, \dots,d$, and
there exists  $r_0<0$ such that
\begin{equation}
\label{b}
\langle \nabla_x b(t,x)\xi,\xi\rangle\leq r_0|\xi|^2,\qquad\;\, (t,x)\in I\times \Rd,\;\, \xi \in \Rd.
\end{equation}
\end{enumerate}
\end{hyp}

\begin{rmk}
{\rm
Assumption \eqref{b} implies that for
any $[a,b]\subset I$ there exists a positive constant $C_{a,b}$ such that
\begin{equation}
\langle b(t,x),x\rangle\le C_{a,b},\qquad\;\,t\in [a,b],\;\,x\in\Rd.
\label{inequality-b}
\end{equation}
Indeed, condition \eqref{b} is equivalent to
\begin{eqnarray*}
\langle b(t,x)-b(t,y),x-y\rangle\le r_0|x-y|^2,\qquad\;\,t\in I,\;\,x,y\in\Rd.
\end{eqnarray*}
Taking $y=0$ and observing that $b$ is continuous, we get
\begin{eqnarray*}
\langle b(t,x),x\rangle\le \|b(\cdot,0)\|_{L^{\infty}(a,b)}|x|+r_0|x|^2,\qquad\;\,t\in [a,b],\;\,x\in\Rd,
\end{eqnarray*}
which implies \eqref{inequality-b} since $r_0<0$.
}
\end{rmk}

Hypotheses \ref{hyp1} yield the existence of
a Markov evolution operator $G(t,s)$ and a {\emph {unique} }(\cite[Rem. 2.8]{AngLorLun}) tight evolution system of
measures $\{\mu_t: t\in I\}$ associated to the evolution operator
$G(t,s)$ (where tight means that for any $\varepsilon>0$ there exists $R>0$
such that $\mu_t(B(0,R))\ge 1-\varepsilon$ for any $t\in I$). More precisely for every $s\in I$ and $f \in C_b(\Rd)$, $G(\cdot,s)f$ is
the unique bounded classical solution  of the Cauchy problem
\begin{eqnarray*}
\left\{
\begin{array}{ll}
D_tu(t,x)={\mathcal{A}}(t)u(t,x),\quad\quad & (t,x)\in (s,+\infty)\times \Rd,\\[1mm]
u(s,x)= f(x),\quad\quad & x\in \Rd.
\end{array}\right.
\end{eqnarray*}
Moreover, $G(\cdot,s)f$ belongs to $C_b([s,+\infty)\times \Rd)\cap C^{1+\alpha/2,2+\alpha}_{\rm loc}((s,+\infty)\times \R^d)$
and it can be represented by
\begin{equation}\label{rep_nucleo}
(G(t,s)f)(x)=\int_{\R^d}f(y)g_{t,s}(x,y)dy,\qquad\;\,x\in\R^d,
\end{equation}
for every $x\in \Rd$, $t>s\in I$ and $f\in C_b(\Rd)$. In \eqref{rep_nucleo},
 $g_{t,s}:\R^d\times\R^d\to\R$ is a positive function
 such that $\|g_{t,s}(x,\cdot)\|_1=1$ for every $t>s\in I$ and $x\in\Rd$
(\cite[Prop. 2.4]{KunLorLun09Non}).

From formula \eqref{rep_nucleo} the following result, which is extensively used in the paper, follows at once.
\begin{lemm}
\label{lemm-prel}
For any $I\ni s<t$ and any nonnegative and non identically vanishing function $f\in C_b(\Rd)$, $G(t,s)f$ is everywhere positive in $\Rd$.
In particular, $|G(t,s)g|\le G(t,s)|g|$ for any $g\in C_b(\Rd)$.
\end{lemm}

By Lemma \ref{lemm-prel}, formula
 \eqref{invariance_intro} and the density of $C_b(\Rd)$ in $L^p(\R^d, \mu_s)$ we have
\begin{eqnarray*}
\|G(t,s)f\|_{p,\mu_t}  \leq \| f\|_{p,\mu_s},
\end{eqnarray*}
for every $t>s$, $p\in [1,+\infty)$ and $f\in L^p(\Rd,\mu_s)$.
Therefore, $G(t,s)$ may be extended to
a contraction (still denoted by $G(t,s)$)  from  $L^p(\R^d, \mu_s)$ to $L^p(\R^d, \mu_t)$.

The dissipativity condition \eqref{b} yields the pointwise gradient estimate
\begin{equation}\label{grad_est_punt}
|(\nabla_x G(t,s)f)(x)|^p\leq e^{pr_0(t-s)}(G(t,s)|\nabla f|^p)(x),
\end{equation}
which holds for every $f \in C^1_b(\Rd)$, $t\geq s$, $x \in \Rd$ and
$p\in [1,+\infty)$ (\cite[Thm. 4.5]{KunLorLun09Non}).

Other remarkable (smoothing) properties of the evolution operator $G(t,s)$ and of the associated evolution system of measures $\{\mu_t: t\in I\}$, which are extensively used in this paper, are stated in the
following two propositions and they can be proved assuming only Hypotheses \ref{hyp1}(i)-(iii).

\begin{prop}[{\cite[Lemma 3.2]{KunLorLun09Non}}]
\label{deriv}
For any $f\in C^2(\Rd)$, which is constant outside a compact set, and any $t\in I$, the function $G(t,\cdot)f$ is differentiable in $I\cap (-\infty,t]$ and
\begin{eqnarray*}
\frac{d}{ds}G(t,s)f=-G(t,s)\mathcal{A}(s)f,\qquad\;\,s\in I,\;\,s\le t.
\end{eqnarray*}
\end{prop}

\begin{prop}[{\cite[Lemma 3.1]{AngLorLun}}]
\label{lemma-3.1}
Let $[a,b]\subset I$. If $f\in C^{1,2}_b([a,b]\times\Rd)$ is such that $f(r,\cdot)$ is constant outside a compact set $K$ for every $r\in [a,b]$, then the function $r \mapsto \int_{\Rd}f(r,x)d\mu_r(x) $ is continuously differentiable in $[a,b]$ and
\begin{eqnarray*}
\frac{d}{dr}\int_{\Rd}f(r,x) d\mu_r(x) =\int_{\Rd}D_rf(r,x)d\mu_r(x)  -\int_{\Rd} (\mathcal{A}(r)f(r, \cdot))(x)d\mu_r(x) ,
\end{eqnarray*}
for every  $r \in [a,b]$.
\end{prop}

The aim of this paper, as already announced in the introduction, is to study the
smoothing effects of the evolution operator $G(t,s)$ on functions with a certain degree of summability. As in the autonomous case we can distinguish different levels of
regularization as specified in the following definition.

\begin{defi}\label{definitions}
The evolution operator $G(t,s)$ is called:
\begin{enumerate}[\rm (i)]
\item ``supercontractive'' if it maps $L^p(\Rd, \mu_s)$ into $L^q(\Rd,\mu_t)$ for any $1<p<q<+\infty$ and $t>s$, and there exists a positive decreasing function $C_{p,q}:(0,+\infty)\to
(0,+\infty)$, such that $\lim_{r\to 0^+}C_{p,q}(r)=+\infty$ and
\begin{eqnarray*}
\|G(t,s)\|_{p\to q}\le C_{p,q}(t-s),\qquad\;\,I\ni s<t;
\end{eqnarray*}
\item ``ultrabounded'' if it maps $L^p(\Rd, \mu_s)$ into $L^\infty(\Rd)$ for every $p>1$ and $t>s$, and there
exists a decreasing function $C_{p,\infty}:(0,+\infty)\to (0,+\infty)$ such that
$\lim_{r\to 0^+}C_{p,\infty}(r)=+\infty$ and
\begin{equation}\label{est_ultrabdd}
\|G(t,s)f\|_{p \to\infty}\leq C_{p,\infty}(t-s),\qquad\;\,I\ni s<t;
\end{equation}
\item ``ultracontractive'' if it maps $L^p(\Rd, \mu_s)$ into $L^\infty(\Rd)$ for every $p\geq 1$ and \eqref{est_ultrabdd} holds for every $p\ge1$.
\end{enumerate}
\end{defi}

\begin{rmk}{\rm
  \begin{enumerate}[\rm(i)]
    \item
The definitions of supercontractivity, ultraboundedness and ultracontractivity given in
Definition \ref{definitions}, where the functions $C_{p,q}$ depend on $t-s$, seem to be the most natural. Indeed, we recall that, if $T(t)$ is a semigroup, then $T(t-s)$ is an evolution operator and the definitions above are the natural extension of those given in the autonomous case.
\item
The strong Feller property enjoyed by the evolution operator (\cite[Cor. 4.3]{KunLorLun09Non}) states that $G(t,s)$ maps $L^\infty(\Rd)$ into $C_b(\Rd)$ for every $t>s$ and it is a contraction, i.e., for every $f\in L^\infty(\Rd)$
\begin{eqnarray*}
\|G(t,s)f\|_{\infty}\leq \|f\|_{\infty},\qquad \;\,I\ni s<t.
\end{eqnarray*}
         Therefore, if $G(t,s)$ is ultrabounded (resp. ultracontractive) then, in fact, it maps $L^p(\Rd,\mu_s)$ into $C_b(\Rd)$ for every $p>1$ (resp. $p\ge 1$) and $t>s$.
  \end{enumerate}}
\end{rmk}

\noindent
\emph{Throughout the paper, if not otherwise specified, we assume that all the conditions in Hypotheses $\ref{hyp1}$ are satisfied.}

\section{Supercontractivity and LSI$_\varepsilon$}\label{section_super}

In this section we provide two criteria to characterize the supercontractivity of the evolution operator $G(t,s)$
by means of a family of logarithmic Sobolev inequalities.

\subsection{The first criterion}
In this subsection we are devoted to prove the following result.

\begin{thm}\label{super_caract}
The following properties are equivalent.
\begin{enumerate}[\rm (i)]
\item
The evolution operator $G(t,s)$ is supercontractive;
\item
the family of logarithmic
Sobolev inequalities
\begin{equation}\label{fam_log_sob_prop}
\int_{\Rd} f^2\log \left(\frac{|f|}{\|f\|_{2,\mu_s}}\right)\,d\mu_s(x) \le \varepsilon \|\,|\nabla f|\,\|_{2,\mu_s}^2+\beta(\varepsilon)\|f\|_{2,\mu_s}^2
\tag{LSI$_\varepsilon$}
\end{equation}
holds for every $f\in H^1(\Rd,\mu_s)$,\, $s\in I$,\, $\varepsilon >0$ and
some positive decreasing function $\beta:(0,+\infty)\to(0,+\infty)$, blowing up as $\varepsilon\to 0^+$.
\end{enumerate}
\end{thm}

The proof of Theorem \ref{super_caract} is based on the following two propositions.
In the first one, we prove a logarithmic Sobolev
inequality satisfied by the evolution operator $G(t,s)$, namely a LSI type estimate satisfied
by the probability measures
$g_{t,s}(x, y)\,dy$ in place of the invariant measures $\mu_s$.

\begin{prop}
For every $f\in C^1_b(\Rd)$, $p \in [2,+\infty)$ and $t,s \in I$, with $t\ge s$, we have
\begin{align}\label{log_sob_G}
G(t,s)(|f|^p \log |f|^p)\leq &  \frac{p^2 \Lambda}{|r_0|} (1-e^{2r_0(t-s)})G(t,s)(|f|^{p-2}|\nabla f|^2)\nnm\\
&+(G(t,s)|f|^p) \log (G(t,s)|f|^p).
\end{align}
\end{prop}

\begin{proof}
We can limit ourselves to proving \eqref{log_sob_G} for $p=2$. Indeed,
for every $p>2$ and $f\in C^1_b(\Rd)$, the claim can be obtained
applying \eqref{log_sob_G} with $p=2$ to the function $|f|^{p/2}$.
Moreover, it is enough to prove \eqref{log_sob_G}, with $p=2$, for nonnegative functions $f \in C^1_b(\Rd)$ with $\sup_{\Rd} f \leq 1$,
taking into account that (by \eqref{rep_nucleo}) $G(t,s)c=c$ for every $c\in \R$.
To this aim we introduce a standard sequence of cut-off functions
\begin{equation}
\label{thetan}
\theta_n(x)=\eta\left(\frac{|x|}{n}\right), \qquad\;\,x\in \R^d, \;\,n\in \N,
\end{equation}
where $\eta \in C^\infty(\R)$ and $\chi_{(-\infty,1]}\leq \eta \leq \chi_{(-\infty,2]}$.

Fix $x \in \Rd$, $s,t\in I$, with $s\leq t$, and a nonnegative function $f \in C^1_b(\Rd)$ with $\|f \|_{\infty}\leq 1$,
and consider the function
\begin{align*}
F_n(r)&= \{G(t,r)[\theta_n (G(r,s)f)^2\log(G(r,s)f)^2]\}(x),
\qquad\;\, s\leq r \leq t,
\end{align*}
which is well defined by Lemma \ref{lemm-prel}.
For any $s\leq r \leq t$, $F_n(r)$ converges to $F(r)=\{G(t,r)[(G(r,s)f)^2\log(G(r,s)f)^2]\}(x)$ as $n \to +\infty$,
by the monotone convergence theorem (see \eqref{rep_nucleo}). Moreover, since the function
$\theta_n (G(r,s)f)^2\log(G(r,s)f)^2$ belongs to $C^2_b(\Rd)$ for every $r>s$ and it vanishes outside  $B(0,2n)$, by
Proposition \ref{deriv} and the formula
\begin{eqnarray*}
\mathcal{A}(r)(g^2\log g^2)= 2 g(1+\log g^2)\mathcal{A}(r)g+ 2(3+\log g^2)\langle Q(r)\nabla g, \nabla g\rangle,
\end{eqnarray*}
which holds for every positive function $g \in C^2(\Rd)$ and every $r\in I$, we get
\begin{align*}
F'_n(r)=& -\Big\{G(t,r)\Big[ 2\theta_n(3+\log (G(r,s)f)^2)\langle Q(r)\nabla_x G(r,s)f, \nabla_x G(r,s)f\rangle\nnm\\
&\quad \,\,\quad  \,\,\quad  \,\,\quad + (G(r,s)f)^2\log(G(r,s)f)^2 \mathcal{A}(r)\theta_n\nnm\\
&\quad  \,\,\quad  \,\,\quad  \,\,\quad + 4(G(r,s)f)(\log(G(r,s)f)^2+1) \langle Q(r)\nabla \theta_n, \nabla_x G(r,s)f\rangle\Big]
\Big\}(x)\\
=&\! : I_{1,n}(r)+I_{2,n}(r)+I_{3,n}(r),
\end{align*}
for any $r\in [s,t]$.
Using the dominated convergence theorem, we have
\begin{eqnarray*}
\lim_{n\to +\infty}I_{1,n}(r)=-2\{G(t,r)[(3+\log (G(r,s)f)^2)\langle Q(r)\nabla_x G(r,s)f, \nabla_x G(r,s)f\rangle ]\}(x).
\end{eqnarray*}
Similarly, since $\nabla\theta_n$ vanishes uniformly in $\Rd$, as $n\to +\infty$, we easily conclude that
$I_{3,n}(r)$ tends to $0$ as $n\to +\infty$.
Now, let us consider the term $I_{2,n}$; by \eqref{inequality-b}, we can estimate
\begin{eqnarray}\label{Artheta}
({\mathcal A}(r)\theta_n)(x)\ge \frac{1}{n^2}\left [C_{s,t}\eta'\left (\frac{|x|}{n}\right )-C\right ],
\end{eqnarray}
for any $n\in\N$, $x\in \Rd$ and any $r\in [s,t]$, where $C=d\Lambda (2\|\eta'\|_{\infty}+\|\eta''\|_{\infty})$. Therefore, recalling that $(G(r,s)f)^2\log(G(r,s)f)^2\le 0$, we conclude that
\begin{eqnarray*}
\liminf_{n\to +\infty}I_{2,n}(r)\ge 0.
\end{eqnarray*}
Summing up, we have proved that
\begin{align*}
\liminf_{n\to +\infty}F'_n(r)\ge &-2\{G(t,r)[(3+\log (G(r,s)f)^2)\langle Q(r)\nabla_x G(r,s)f, \nabla_x G(r,s)f\rangle ]\}(x)\nnm\\
\ge & -6\Lambda\{G(t,r)(|\nabla_x G(r,s)f|^2)\}(x),
\end{align*}
for any $r\in [s,t]$ and any $x\in\Rd$.
Since the sequence $F_n'$ is bounded from below by a constant, from the Fatou lemma we can conclude that
\begin{align*}
F(t)-F(s)&=\lim_{n\to +\infty}(F_n(t)-F_n(s))\notag\\
&\ge \int_s^t\liminf_{n\to +\infty}F_n'(\sigma)d\sigma\notag\\
&\ge -6\Lambda\int_s^t\{G(t,r)(|\nabla_x G(r,s)f|^2)\}(x)dr.
\end{align*}
Using the gradient estimate \eqref{grad_est_punt} we get
\begin{align*}
F(t)-F(s)\ge& -6\Lambda (G(t,s)|\nabla f|^2)(x)\int_s^te^{2r_0(r-s)}dr\nnm\\
=&\frac{3\Lambda}{|r_0|}(e^{2r_0(t-s)}-1)(G(t,s)|\nabla f|^2)(x),
\end{align*}
and \eqref{log_sob_G} follows.
\end{proof}

Next proposition shows that the boundedness of $G(t,s)$ from $L^p(\Rd,\mu_s)$ into $L^q(\Rd,\mu_t)$, for any $t>s$, yields a family of logarithmic
Sobolev inequalities satisfied by the system of invariant measures $\{\mu_t: \,t\in I\}$.
The key tools used in the proof are estimate \eqref{log_sob_G}
and the Riesz-Thorin's interpolation theorem.

\begin{prop}\label{super}
Assume that, for every $s\in I$, $t>s$ and $1<p<q<+\infty$, $\tilde C_{p,q}(t,s):=\|G(t,s)\|_{p\to q}<+\infty$. Then,
\begin{align}\label{super_LSI}
\int_{\Rd}f^2 \log \left(\frac{|f|}{\|f\|_{2,\mu_s}} \right)d\mu_s(x) \leq &\frac{2\Lambda p(q-1)}{|r_0|(q-p)}(1-e^{2r_0(t-s)})\|\,|\nabla f|\,\|_{2,\mu_s}^2\nnm\\
&+ \frac{pq}{2(q-p)}\log(\tilde C_{p,q}(t,s))\|f\|_{2,\mu_s}^2,
\end{align}
for every $s\in I,\, t>s,\, f\in H^1(\Rd,\mu_s)$, where $r_0$ is the constant in $\eqref{b}$.
\end{prop}

\begin{proof}
The proof can be obtained adapting the arguments in the proof of \cite[Thm. 2.1(1)]{RocWan03Sup}. For the reader's convenience we enter into details.

We split the proof into two steps. In the first one we show that it suffices to prove \eqref{super_LSI} for functions $f\in C^1_c(\Rd)$ such
that $\|f\|_{2,\mu_s}=1$. In the second
step, we get estimate \eqref{super_LSI} for such functions.

{\em Step 1.}
For notational convenience, we set
\begin{eqnarray*}
M_1(t,s)=\frac{2\Lambda p(q-1)}{|r_0|(q-p)}(1-e^{2r_0(t-s)}),\qquad\;\,
M_2(t,s)=\frac{pq}{2(q-p)}\log(\tilde C_{p,q}(t,s)).
\end{eqnarray*}
We assume that inequality \eqref{super_LSI} holds for any function $f \in C^1_c(\Rd)$ such that $\|f\|_{2,\mu_s}=1$, and we show that it actually
holds for any $f\in H^1(\Rd,\mu_s)$. For this purpose, let $f\in H^{1}(\Rd,\mu_s)$ satisfy $\|f\|_{2,\mu_s}=1$, and consider a sequence $(f_n)_n\in C^1_c(\Rd)$ such that
$\|f_n-f\|_{H^1(\Rd,\mu_s)}$ tends to $0$ as $n \to +\infty$ (see \cite[Lemma 2.5]{AngLorLun}).
Without loss of generality, we can assume that $\|f_n\|_{2,\mu_s}=1$ for any $n\in\N$.
Up to a subsequence, $f_n(x)$ converges to $f(x)$ for almost every $x\in\Rd$ as $n \to +\infty$ and
\begin{align*}
\int_{\Rd}f_n^2 \log |f_n|d\mu_s(x) \leq M_1(t,s)\int_{\Rd}|\nabla f_n|^2d\mu_s(x) + M_2(t,s),
\end{align*}
for every $n\in \N$.
Let us split $f_n^2\log |f_n|= f_n^2\log_{+} |f_n|- f_n^2|\log |f_n|\,|\chi_{\{|f_n|\le 1\}}$, where $\log_+(r)=\max\{\log(r),0\}$ for any $r>0$.
Since $f_n^2|\log|f_n|\,|\le(2e)^{-1}=\sup_{x\in (0,1]}x^2|\log x|$ for any $n\in\N$, the dominated convergence theorem yields
\begin{eqnarray*}
\lim_{n \to +\infty}\int_{\Rd} f_n^2|\log |f_n|\,|\chi_{\{|f_n|\le 1\}}\,d\mu_s(x) =\int_{\Rd} f^2|\log |f|\,|\chi_{\{|f|\le 1\}}\,d\mu_s(x) .
\end{eqnarray*}
Thus, by Fatou lemma we deduce that
\begin{align*}
&\int_{\Rd}f^2\log_+|f|\,d\mu_s(x) \\
\le &\liminf_{n\to +\infty}\left (M_1(t,s)\int_{\Rd}|\nabla f_n|^2\,d\mu_s(x) +M_2(t,s)+
\int_{\Rd}f_n^2|\log |f_n|\,|\chi_{\{|f_n|\le 1\}}\,d\mu_s(x) \right )\\
=&M_1(t,s)\int_{\Rd}|\nabla f|^2\,d\mu_s(x) +M_2(t,s)+
\int_{\Rd}f^2|\log |f|\,|\chi_{\{|f|\le 1\}}\,d\mu_s(x) ,
\end{align*}
which leads immediately to \eqref{super_LSI}.

Finally, the condition $\|f\|_{2,\mu_s}= 1$ can be removed applying
\eqref{super_LSI} to the function $f(\|f\|_{2,\mu_s})^{-1}$.

{\em Step 2.}
Let us prove the claim for $f \in C^1_c(\Rd)$ such that $\|f\|_{2,\mu_s}=1$.
The starting point is formula \eqref{log_sob_G} with $p=2$ which yields
\begin{equation}\label{log_sob_g_2}
G(t,s)(f^2\log f^2)\leq  \frac{4 \Lambda}{|r_0|} (1-e^{2r_0(t-s)})G(t,s)|\nabla f|^2+(G(t,s)f^2) \log (G(t,s)f^2),
\end{equation}
for any $s,t\in I$, with $s<t$ and any $f\in C^1_c(\Rd)$. Integrating \eqref{log_sob_g_2} in $\Rd$ with respect to the measure $\mu_t$
and using \eqref{invariance_intro}, we get
\begin{align}\label{log_sob_g_2_integrate}
\int_{\Rd}f^2\log f^2d\mu_s(x) \leq  &\frac{4\Lambda}{|r_0|} (1-e^{2r_0(t-s)})\int_{\Rd}|\nabla f|^2d\mu_s(x) \nnm\\
&+\int_{\Rd}(G(t,s)f^2) \log (G(t,s)f^2)d\mu_t(x) .
\end{align}

Let us fix $1<p<q<+\infty$. By assumptions, $\|G(t,s)\|_{p \to q}=\tilde C_{p,q}(t,s)<+\infty$,
for every $t,s \in I$ such that $t>s$.
Since $\|G(t,s)\|_{1\to 1}\leq 1$,
from the Riesz-Thorin's interpolation theorem we get that
\begin{equation}
\label{ticito}
\|G(t,s)f\|_{q_h,\mu_t}\leq (\tilde C_{p,q}(t,s))^{r_h}\|f\|_{p_h,\mu_s},
\end{equation}
for every $f \in L^p(\Rd,\mu_s)$ and $h\in (0,1-1/p)$, where
\begin{eqnarray*}
r_h=\frac{ph}{p-1}\in (0,1),\qquad\;\, \frac{1}{p_h}=1-r_h+\frac{r_h}{p}, \qquad\;\,\frac{1}{q_h}= 1-r_h+\frac{r_h}{q}.
\end{eqnarray*}
Fix $f\in C_b(\Rd)$ such that $\|f\|_{2,\mu_s}=1$. Then, from \eqref{ticito} and, since $p_h=(1-h)^{-1}$, we have
\begin{eqnarray*}
\int_{\Rd} (G(t,s)|f|^{2(1-h)})^{q_h}d\mu_t(x)\leq (\tilde C_{p,q}(t,s))^{r_h q_h},\qquad\;\, t>s,
\end{eqnarray*}
which holds also for $h=0$. Consequently,
\begin{align}
&\frac{1}{h}\left (\int_{\Rd} (G(t,s)|f|^{2(1-h)})^{q_h}d\mu_t(x)-\int_{\Rd}G(t,s)|f|^2d\mu_t(x)\right )\notag\\
=&\frac{1}{h}\left (\int_{\Rd} (G(t,s)|f|^{2(1-h)})^{q_h}d\mu_t(x)-1\right )\notag\\
\le &\frac{1}{h}\left (\tilde C_{p,q}(t,s))^{r_h q_h}-1\right ).
\label{increm-ratio}
\end{align}
The first and the last sides of \eqref{increm-ratio} represent respectively the incremental ratio at $h=0$ of the functions
$h\mapsto \|G(t,s)|f|^{2(1-h)}\|_{q_h,\mu_s}^{q_h}$ and $h\mapsto (\tilde C_{p,q}(t,s))^{r_hq_h}$. Since these two functions are differentiable at $h=0$, we
immediately deduce that
\begin{align*}
&\frac{p(q-1)}{q(p-1)}\int_{\Rd} G(t,s)f^2\log(G(t,s)f^2) d\mu_t(x)-\int_{\Rd} G(t,s)(f^2\log f^2) d\mu_t(x)\nnm\\
\leq &\frac{p}{p-1}\log (\tilde C_{p,q}(t,s)),
\end{align*}
or, equivalently, since $\{\mu_t: t\in I\}$ is an evolution system of measure,
\begin{align*}
\int_{\Rd} G(t,s)f^2\log(G(t,s)f^2)d\mu_t(x)\leq &\frac{q(p-1)}{p(q-1)}\int_{\Rd}f^2\log f^2d\mu_s(x)\nnm\\
 &+ \frac{q}{q-1}\log (\tilde C_{p,q}(t,s)),
\end{align*}
which, replaced into \eqref{log_sob_g_2_integrate}, yields
\begin{align*}
\int_{\Rd}f^2\log |f|d\mu_s(x) \leq  &\frac{2\Lambda}{|r_0|}\frac{p(q-1)}{q-p}(1-e^{2r_0(t-s)})\int_{\Rd}|\nabla f|^2d\mu_s(x) \nnm\\
&+\frac{pq}{2(q-p)}\log(\tilde C_{p,q}(t,s)),
\end{align*}
and the claim is proved.
\end{proof}

\begin{proof}[Proof of Theorem $\ref{super_caract}$]
``$(i)\Rightarrow (ii)$''
By Proposition \ref{super}, if $G(t,s)$ is supercontractive, then
the following family of logarithmic Sobolev inequalities
\begin{equation}\label{log_r}
\int_{\Rd}f^2 \log f^2 d\mu_s(x) \leq r(t-s) \int_{\Rd}|\nabla f|^2d\mu_s(x) + \widetilde{\beta}(t-s),
\end{equation}
holds for every $s\in I,\,t>s$, and $f\in C^1_b(\Rd)$ with $\|f\|_{2,\mu_s}=1$. Since
$\tilde C_{p,q}(t,s)\le C_{p,q}(t-s)$, in formula \eqref{log_r} we have
\begin{eqnarray*}
r(t-s)=\frac{4\Lambda p(q-1)}{|r_0|(q-p)}(1-e^{2r_0(t-s)}),\qquad\;\,
\widetilde{\beta}(t-s)=\frac{pq}{q-p}\log(C_{p,q}(t-s)).
\end{eqnarray*}
$\tilde\beta$ is a positive function defined in $(0,\infty)$ and $2\leq p\leq q$.

Inverting the function $r$  we obtain
\begin{eqnarray*}
t-s=\frac{1}{2r_0}\log\left (1+\frac{r_0(q-p)}{4\Lambda(q-1)}r\right ),\qquad\;\,r\in [0,\overline{r}),
\end{eqnarray*}
where $\overline{r}=\frac{4\Lambda p(q-1)}{|r_0|(q-p)}$.
Thus \eqref{fam_log_sob_prop} holds for every $\varepsilon\in (0, \overline{r})$ with
\begin{eqnarray*}
\beta(\varepsilon)=\frac{pq}{q-p}\log\left[ C_{p,q}\left (\frac{1}{2r_0}\log\left (1+\frac{r_0(q-p)}{4\Lambda(q-1)}\varepsilon\right )\right )\right ].
\end{eqnarray*}
Clearly we can extend \eqref{fam_log_sob_prop} to any $\varepsilon>0$ and any $f\in H^1(\Rd,\mu_s)$ by setting $\beta(\varepsilon)=\lim_{r\to \overline{r}^-}\beta(r)$ for any $\varepsilon\geq \overline{r}$, and
using a standard approximation argument.
\smallskip

``$(ii)\Rightarrow (i)$'' Assume that estimate \eqref{fam_log_sob_prop} holds for every $f\in H^1(\Rd,\mu_s)$,\, $s\in I$,\, $\varepsilon >0$ and
some positive decreasing function $\beta:(0,+\infty)\to(0,+\infty)$. Then, for every $f\in C^1_c(\Rd)$, $p\in(1,+\infty)$ and $s\in I$, writing \eqref{fam_log_sob_prop} for the function $|f|^{p/2}$, we have
\begin{align}\label{fam_p}
\int_{\Rd} |f|^p\log \left(\frac{|f|}{\|f\|_{p,\mu_s}}\right)\,d\mu_s(x)
\le &\varepsilon \frac{p}{2} \int_{\Rd}|f|^{p-2}|\nabla f|^2d\mu_s(x) +\frac{2\beta(\varepsilon)}{p}\|f\|_{p,\mu_s}^p.
\end{align}
Using \eqref{fam_p}  we deduce supercontractivity of $G(t,s)$.
Indeed, let $\varepsilon>0$, $f\in C^1_c(\Rd)$ be nonnegative and non identically vanishing in $\Rd$, $p\in(1,+\infty)$ and $s\in I$; we set
\begin{equation}\label{q_m}
q(t):=e^{2\eta_0\varepsilon^{-1}(t-s)}(p-1)+1,\qquad\;\, m(t):=2\beta(\varepsilon)\left(p^{-1}-(q(t))^{-1}\right),
\end{equation}
\begin{eqnarray*}
H(t):=e^{-m(t)}\left(\int_{\Rd}(G(t,s)f)^{q(t)}\,d\mu_t(x) \right)^{1/q(t)},
\end{eqnarray*}
for any $t\ge s$.
To prove that $G(t,s)$ is supercontractive, we show that $H$ is a non increasing function. We would like to differentiate the function $H$ and show that its derivative is nonpositive in $(s,+\infty)$. Unfortunately,
we can differentiate functions of the type
$t\mapsto \int_{\Rd}\psi d\mu_t$ only when $\psi$ is constant outside a compact set, which, in general, is not our case.
For this purpose we use an approximation argument and introduce the functions
$H_n$ ($n\in\N$) defined by
\begin{eqnarray*}
H_n(t):=e^{-m(t)}\left(\int_{\Rd}\theta_n (G(t,s)f)^{q(t)}\,d\mu_t(x) \right)^{1/q(t)},
\end{eqnarray*}
where $\theta_n$ is defined in \eqref{thetan}.
From Proposition \ref{lemma-3.1}, for every $n \in \N$, the function $H_n$ is differentiable for $t>s$ with derivative given by
\begin{eqnarray*}
H'_n(t)=H_n(t)(-m'(t)+\varphi_n(t)),\quad\,\,t>s,
\end{eqnarray*}
where
\begin{align*}
\varphi_n&(t)  = \left(\int_{\Rd}\theta_n (G(t,s)f)^{q(t)}d\mu_t(x) \right)^{-1}\\
&\times\bigg\{\frac{q'(t)}{q(t)}\int_{\Rd}\theta_n(G(t,s)f)^{q(t)}\log (G(t,s)f)d\mu_t(x) \\
&\qquad-\frac{q'(t)}{(q(t))^2}\left(\int_{\Rd}\theta_n (G(t,s)f)^{q(t)}d\mu_t(x) \right)\log\left (\int_{\Rd}\theta_n (G(t,s)f)^{q(t)}d\mu_t(x) \right)\\
&\qquad- (q(t)-1)\int_{\Rd}\theta_n(G(t,s)f)^{q(t)-2}\langle Q(t)\nabla_x G(t,s)f,\nabla_x G(t,s)f \rangle d\mu_t(x) \\
&\qquad-\frac{2}{q(t)}\int_{\Rd}\langle Q(t)\nabla_x((G(t,s)f)^{q(t)}),\nabla\theta_n\rangle d\mu_t(x) \\
&\qquad-\frac{1}{q(t)}\int_{\Rd}(G(t,s)f)^{q(t)}{\mathcal A}(t)\theta_n d\mu_t(x)  \bigg\} .
\end{align*}
Using \eqref{Artheta} we can show that $\limsup_{n\to +\infty}\varphi_n(t)\le \psi(t)$ for every $t>s$, where
\begin{align*}
\psi(t)
:=&\left(\int_{\Rd}(G(t,s)f)^{q(t)}d\mu_t(x) \right)^{-1}\bigg\{\frac{q'(t)}{q(t)}\int_{\Rd}(G(t,s)f)^{q(t)}\log (G(t,s)f)d\mu_t(x) \\
&-\frac{q'(t)}{(q(t))^2}\left(\int_{\Rd} (G(t,s)f)^{q(t)}d\mu_t(x) \right)\log\left (\int_{\Rd} (G(t,s)f)^{q(t)}d\mu_t(x) \right)\\
&- (q(t)-1)\int_{\Rd}(G(t,s)f)^{q(t)-2}\langle Q(t)\nabla_x G(t,s)f,\nabla_x G(t,s)f \rangle d\mu_t(x) \bigg\}.
\end{align*}
Writing
\begin{eqnarray*}
H_n(t)-H_n(s)=\int_s^tH_n(\sigma)(-m'(\sigma)+\varphi_n(\sigma))d\sigma
\end{eqnarray*}
and letting $n\to +\infty$ yields
\begin{equation}\label{H}
H(t)-H(s)\leq \int_s^t H(\sigma)(-m'(\sigma)+\psi(\sigma))\,d\sigma.
\end{equation}
From \eqref{ell} we get
\begin{align*}
-m'(\sigma)+\psi(\sigma)&\leq[e^{m(\sigma)}H(\sigma)]^{-q(\sigma)}\frac{q'(\sigma)}{q(\sigma)}\bigg\{
-m'(\sigma)\frac{q(\sigma)}{q'(\sigma)}\int_{\Rd}(G(\sigma,s)f)^{q(\sigma)}d\mu_\sigma(x)\\
&+\int_{\Rd}(G(\sigma,s)f)^{q(\sigma)}\log (G(\sigma,s)f)d\mu_\sigma(x) \\
&-\left(\int_{\Rd} (G(\sigma,s)f)^{q(\sigma)}d\mu_\sigma(x) \right)\log\left (\int_{\Rd} (G(\sigma,s)f)^{q(\sigma)}d\mu_\sigma(x) \right)^{\frac{1}{q(\sigma)}}\\
&-\eta_0 \frac{q(\sigma)(q(\sigma)-1)}{q'(\sigma)}\int_{\Rd}(G(\sigma,s)f)^{q(\sigma)-2}|\nabla_x G(\sigma,s)f|^2 d\mu_\sigma(x)\bigg\}.
\end{align*}
Now, applying the logarithmic Sobolev inequality \eqref{fam_p} with $G(\sigma,s)f$ and
$q(\sigma)$ in place of $f$ and $p$ respectively, we get
\begin{align*}
-m'(\sigma)+\psi(\sigma)&\leq[e^{m(\sigma)}H(\sigma)]^{-q(\sigma)}\frac{q'(\sigma)}{q(\sigma)}\bigg\{\left(\varepsilon \frac{q(\sigma)}{2}-\eta_0 \frac{q(\sigma)(q(\sigma)-1)}{q'(\sigma)}\right)\\
& \qquad\quad\times\int_{\Rd}(G(\sigma,s)f)^{q(\sigma)-2}|\nabla_x G(\sigma,s)f|^2d\mu_\sigma(x) \\
& \quad+\left(2\frac{\beta(\varepsilon)}{q(\sigma)}-m'(\sigma)\frac{q(\sigma)}{q'(\sigma)}\right)\int_{\Rd}(G(\sigma,s)f)^{q(\sigma)}d\mu_\sigma(x)
\bigg\}\\
& =0,
\end{align*}
by the definition of $q$ and $m$ given in \eqref{q_m}. Therefore, from \eqref{H} we deduce that $H(t)\leq H(s)$, so that $H$ is nonincreasing, i.e,
\begin{equation}
\|G(t,s)f\|_{q(t),\mu_t}\leq e^{2\beta(\varepsilon)\left(\frac{1}{p}-\frac{1}{q(t)}\right)}\|f\|_{p,\mu_s}.
\label{super-contr}
\end{equation}
Now, for any $q>p$ and $t>s$, we fix $\varepsilon=2\eta_0(t-s)(\log((q-1)/(p-1)))^{-1}$. We thus deduce that $q(t)=q$ and, from \eqref{super-contr}
we obtain
\begin{eqnarray*}
\|G(t,s)f\|_{q,\mu_t}\leq C_{p,q}(t-s)\|f\|_{p,\mu_s},
\end{eqnarray*}
with
\begin{eqnarray*}
C_{p,q}(r)=\exp\left [\frac{2(q-p)}{pq}\beta\left (2\eta_0r\left (\log\left (\frac{q-1}{p-1}\right )\right )^{-1}\right )\right ],\qquad\;\,r>0,
\end{eqnarray*}
which is a decreasing function since $\beta$ is decreasing as well.

The density of $C^1_c(\Rd)$ in $L^p(\Rd,\mu_s)$ allows us to complete the proof.
\end{proof}

\subsection{A second criterion}
Here we show that the integrability  with respect to the measures $\{\mu_t:t\in I\}$ (uniform in $t$) of the Gaussian functions
$\varphi_{\lambda}(x):=e^{\lambda |x|^2}$ for every $\lambda>0$ is another condition equivalent to the supercontractivity of $G(t,s)$.
To this aim we first prove some preliminary results.
The first proposition, whose proof is an adaption of Ledoux's method \cite{Led95Rem} to our setting, yields some exponential integrability result. A more general result than next Proposition \ref{integrability}
has been proved
in \cite{GroRot98Her}, in the autonomous setting still assuming the validity of the \eqref{fam_log_sob_prop}, where
the evolution system of measures is replaced by a unique invariant measure.

\begin{prop}\label{integrability}
The function
$x\mapsto e^{\lambda |x|}$ belongs to $L^1(\Rd,\mu_s)$ for every $\lambda>0$.
More precisely,
\begin{eqnarray*}
\sup_{s\in I}\int_{\Rd}e^{\lambda |x|}d\mu_s(x) <+\infty,\qquad\;\,\lambda>0.
\end{eqnarray*}
Moreover, if the inequality \eqref{fam_log_sob_prop} holds, then
$\varphi_{\lambda}\in L^1(\Rd,\mu_s)$ for every $\lambda>0$ and
\begin{eqnarray*}
\sup_{s\in I}\int_{\Rd}\varphi_{\lambda}(x)d\mu_s(x) <+\infty,\qquad\;\,\lambda>0.
\end{eqnarray*}
\end{prop}
\begin{proof}
For every $n \in \N$, let $\psi_n:[0,+\infty)\to\R$ be a smooth increasing function such that $\psi_n(t)=t$ for any $t\in [0,n]$,
$\psi_n(t)=n+1$ for any $t\ge n+2$ and $0\le\psi_n'(t)\le 1$ for any $t\ge 0$.
The functions $f_n(x):=\psi_n(|x|)$ are bounded and satisfy $\|\,|\nabla f_n|\,\|_{\infty}\leq 1$ for any $n\in\N$. Moreover,
$f_n(x)$ converges increasingly to $f(x):=|x|$ for any $x\in\Rd$, as $n \to +\infty$.
Fix $s\in I$, $\lambda>0$ and $n\in\N$. We set $H_{n,\lambda}(r):=\displaystyle{\int_{\Rd}} e^{\lambda r f_n}d\mu_s(x) $ for any $r>0$, and observe that
\begin{equation}\label{equality}
H_{n,\lambda}'(r)= \lambda \int_{\Rd}e^{\lambda r f_n}f_nd\mu_s(x) .
\end{equation}
Applying the logarithmic Sobolev inequality \eqref{LSI} to the function $e^{\lambda r f_n/2}$ and using \eqref{equality}, we get
\begin{equation*}
r H_{n,\lambda}'(r)-H_{n,\lambda}(r)\log H_{n,\lambda}(r)\leq \frac{C\lambda^2r^2}{2}\int_{\Rd}e^{\lambda r f_n}|\nabla f_n|^2d\mu_s(x) \leq \frac{C\lambda^2 r^2}{2} H_{n,\lambda}(r),
\end{equation*}
for every $n \in \N$. Now, dividing by $r^2 H_{n,\lambda}(r)$ we have
\begin{equation}
\label{est_1_first}
\left(\frac{1}{r}\log H_{n,\lambda}(r)\right)'=\frac{1}{r}\frac{H_{n,\lambda}'(r)}{ H_{n,\lambda}(r)}-\frac{1}{r^2}\log H_{n,\lambda}(r)\leq \frac{C\lambda^2}{2}.
\end{equation}
Integrating \eqref{est_1_first} from $1$ to $2$ with respect to $r$ we deduce that
\begin{eqnarray*}
H_{n,\lambda}(2)\leq e^{C\lambda^2}(H_{n,\lambda}(1))^2.
\end{eqnarray*}
Since the evolution system of measures $\{\mu_t: t\in I\}$ is tight, we can choose $M>0$ such that
$\mu_s(\Rd\setminus B(0,M\lambda^{-1}))\leq (4e^{C\lambda^2})^{-1}$ for every $s\in I$.
This fact and the monotonicity of $\psi_n$ imply that
\begin{align*}
\mu_s(\{\lambda f_n\geq M\})
\le \mu_s(\{\lambda f\geq M\})
=\mu_s(\Rd\setminus B(0,M\lambda^{-1}))\le
(4e^{C\lambda^2})^{-1},
\end{align*}
for every $s\in I$.
Now,
\begin{align*}
\int_{\Rd} e^{\lambda f_n}d\mu_s(x) &=\int_{\{\lambda f_n \geq M\}} e^{\lambda f_n}d\mu_s(x) +\int_{\{\lambda f_n< M\}} e^{\lambda f_n}d\mu_s(x) \\
& \leq (\mu_s(\{\lambda f_n\geq M\}))^{\frac{1}{2}}\left(\int_{\Rd}e^{2\lambda f_n}d\mu_s(x) \right)^{\frac{1}{2}}+e^M\\
& \leq (4e^{C\lambda^2})^{-\frac{1}{2}} \big(H_{n,\lambda}(2)\big)^{\frac{1}{2}}+e^M\\
& \leq 2^{-1}H_{n,\lambda}(1)+e^M.
\end{align*}
Hence, $\int_{\Rd}e^{\lambda f_n}d\mu_s(x) \leq 2 e^M$ for any $s\in I$, and letting $n\to+\infty$ we get the first part of the claim.

In order to prove the second part of the claim assume that \eqref{fam_log_sob_prop} holds and, for brevity, we set $H_n:=H_{n,1}$.
Arguing as before and applying \eqref{fam_log_sob_prop} to $e^{r f_n/2}$, we get
\begin{equation}
\label{est_1}
\left(\frac{1}{r}\log H_n(r)\right)'=\frac{1}{r}\frac{H_n'(r)}{ H_n(r)}-\frac{1}{r^2}\log H_n(r)\leq \frac{\varepsilon}{2}+2\frac{\beta(\varepsilon)}{r^2},
\end{equation}
for every $\varepsilon>0$ and $n \in \N$. Integrating \eqref{est_1} from $\gamma$ to $\sigma$ we deduce that
\begin{eqnarray*}
\frac{1}{\sigma}\log H_n(\sigma)-\frac{1}{\gamma}\log H_n(\gamma)\leq\frac{\varepsilon}{2}(\sigma-\gamma)+2\beta(\varepsilon)\left(\frac{1}{\gamma}-\frac{1}{\sigma}\right).
\end{eqnarray*}
Therefore, for every $0<\gamma<\sigma$ and $\varepsilon>0$,
\begin{equation}\label{est_hdelta}
H_n(\sigma)\leq \exp\left (\frac{\varepsilon}{2}\sigma^2+\sigma\left(\frac{\log H_n(\gamma)}{\gamma}-\frac{\varepsilon}{2}\gamma+\frac{2}{\gamma}\beta(\varepsilon)\right)-2\beta(\varepsilon)\right ).
\end{equation}
Now, we observe that
\begin{eqnarray*}
2\sqrt{\lambda\pi}\int_{\Rd}e^{\lambda f_n^2}d\mu_s(x) =\int_{\Rd}\int_\R e^{\sigma f_n-\frac{\sigma^2}{4\lambda}}\,d\sigma\,d\mu_s(x) = \int_\R H_n(\sigma)e^{-\frac{\sigma^2}{4\lambda}}d\sigma.
\end{eqnarray*}
Moreover, by \eqref{est_hdelta}
\begin{equation}
\int_\R H_n(\sigma)e^{-\frac{\sigma^2}{4\lambda}}d\sigma \leq \int_\R e^{\sigma^2(\frac{\varepsilon}{2}-\frac{1}{4\lambda})+\sigma\left(\log\|e^{\gamma |\cdot|}\|_{1,\mu_s}^{1/\gamma}-\frac{\varepsilon}{2}\gamma+\frac{2}{\gamma}\beta(\varepsilon)\right)}\,d\sigma,
\label{3.15}
\end{equation}
which is finite for every $0<\lambda<\frac{1}{2\varepsilon}$ and $n\in \N$.
By the arbitrariness of $\varepsilon$ and observing that
$\sup_{s\in I}\|e^{\gamma |\cdot|}\|_{1,\mu_s}^{1/\gamma}<+\infty$, by the first part of the proof, we deduce that
\begin{equation}\label{lett_n}
\int_{\Rd}e^{\lambda f_n^2(x)}d\mu_s(x) \le K,\qquad\;\, \lambda>0,\;\, n\in \N,
\end{equation}
for some positive constant $K$, independent of $s$. Finally, we get the claim by the monotone convergence theorem letting $n \to +\infty$ in \eqref{lett_n}.
\end{proof}

\begin{rmk}\rm{
\begin{enumerate}[\rm (i)]
\item
Actually, formula \eqref{3.15} shows that, just
assuming the validity of the estimate \eqref{LSI}, one can deduce that the functions $\varphi_\lambda$ belong to $L^1(\Rd,\mu_s)$ and
$\sup_{s\in I}\|\varphi_{\lambda}\|_{1,\mu_s}<+\infty$ for every $\lambda <(2C)^{-1}$
where $C$ is the constant in \eqref{LSI}.
\item
We point out that in the proof of Proposition \ref{integrability} we have not used the invariance of the measures $\{\mu_t:t\in I\}$. 
\end{enumerate}}
\end{rmk}

Next proposition is an Harnack-type estimate satisfied by the evolution operator $G(t,s)$.
The proof of this result is essentially based on the gradient estimates \eqref{grad_est_punt} and
extends the method used in \cite{wang_97} to the nonautonomous case.

\begin{prop}[An Harnack-type inequality]\label{Har_prop}
For every $f \in C_b(\Rd)$, $p>1$, $t>s$ and $x,y \in \Rd$ we have
\begin{equation}\label{Harnack}
|(G(t,s)f)(x)|^p\leq (G(t,s)|f|^p)(y)\exp{\left(\frac{p\,|x-y|^2}{4(p-1)\eta_0(t-s)}\right)}.
\end{equation}
\end{prop}
\begin{proof}
Since $|G(t,s)f|\leq G(t,s)|f|$
for every $f\in C_b(\Rd)$ and $t>s$, it suffices to prove \eqref{Harnack} for nonnegative functions $f$.

We split the proof into two steps. In the first one we prove \eqref{Harnack} for nonnegative functions $f\in C^1_b(\Rd)$.
In the second step, by standard approximation arguments we extend \eqref{Harnack} to every nonnnegative function $f\in C_b(\Rd)$.

{\em Step 1.} Let $f\in C^1_b(\Rd)$ be a nonnegative function. Fix $t>s$, $x,y \in \Rd$ and set
\begin{eqnarray*}
\Phi_n(r):=\{G(t,r)[\theta_n (G(r,s)f)^p]\}(\psi(r)),\qquad\;\, s\leq r\leq t,
\end{eqnarray*}
where $\theta_n$ is the sequence of cut-off functions defined in \eqref{thetan} and
\begin{eqnarray*}
\psi(r):=\left(\frac{t-r}{t-s}\right) y+\left(\frac{r-s}{t-s}\right) x,\qquad\;\, s\leq r\leq t.
\end{eqnarray*}
By Lemma \ref{lemm-prel} and Proposition \ref{deriv}, the function $\log \Phi_n$ is well defined, it belongs to
$C^1((s,t))$ for every $n \in \N$ and there exist $n_0\in \N$ and a
positive constant $C_{\Phi}$ such that $\Phi_n(r)\geq C_{\Phi}$ for every $n>n_0$ and $r \in [s,t]$.
This last assertion follows since $\Phi_n(r)>0$ for every $r<t$ and $\Phi_n(t)=(\theta_n (G(t,s)f)^p)(x)$ for every $n \in \N$.
Hence, choosing $n$ large enough such that $x \in \supp \theta_n$ we conclude.

Differentiating the functions $r \mapsto \log \Phi_n(r)$ ($n\in\N$) in $(s,t)$ we get
\begin{align}\label{der_log}
\frac{d}{dr}\log \Phi_n(r)=\frac{1}{\Phi_n(r)}\bigg\{&
-\{G(t,r)[\mathcal{A}(r)(\theta_n (G(r,s)f)^p)]\}(\psi(r))\notag\\
&+ \{G(t,r)[\theta_n D_r (G(r,s)f)^p]\}(\psi(r))\notag\\
&+\frac{1}{t-s}\langle \nabla_x[G(t,r)(\theta_n (G(r,s)f)^p)](\psi(r)) , x-y \rangle\bigg\}.
\end{align}
Let observe that, if $g=G(\cdot,s)f$, then
\begin{align*}
D_r g^p-\mathcal{A}(r)g^p=-p(p-1)g^{p-2}\langle Q(r)\nabla_x g,\nabla_x g\rangle
\end{align*}
and
\begin{align*}
|\nabla_x G(t,r)(\theta_n g^p)|&\leq G(t,r)|\nabla_x(\theta_n g^p)|\\
&\leq G(t,r)(|\nabla \theta_n| g^p+p\eta_0^{-1/2}\theta_n g^{p-1}|Q^{1/2}(r)\nabla_x g|),
\end{align*}
where in the last inequality we have used \eqref{grad_est_punt} and \eqref{ell}.
From \eqref{der_log} we get
\begin{align*}
&\frac{d}{dr}\log \Phi_n(r)\\
\leq -&\frac{1}{\Phi_n(r)}\Big\{G(t,r)\big [(G(r,s)f)^p \mathcal{A}(r)\theta_n+2p\,(G(r,s)f)^{p-1}\langle Q(r)\nabla\theta_n,\nabla_x G(r,s)f\rangle\\
&\quad\qquad\qquad\qquad+p(p-1)\theta_n(G(r,s)f)^{p-2}\langle Q(r)\nabla_x G(r,s)f,\nabla_x G(r,s)f \rangle\big ]\\
& \quad\qquad\;\; -\frac{|x-y|}{t-s}G(t,r)\big [p\eta_0^{-1/2}\theta_n (G(r,s)f)^{p-1}|Q^{1/2}(r)\nabla_x G(r,s)f|\\
& \qquad\qquad\qquad\qquad\qquad\qquad\, +|\nabla \theta_n|(G(r,s)f)^p\big]\Big\}(\psi(r)),
\end{align*}
hence,
\begin{align*}
&\frac{d}{dr}\log \Phi_n(r)\\
\leq &\frac{1}{\Phi_n(r)}\Big\{G(t,r)\Big[g_n(r)
-p\theta_n(G(r,s)f)^p\Big ((p-1)h^2(r)-\frac{|x-y|}{\sqrt{\eta_0}(t-s)}h(r)\Big )\Big]\Big\}(\psi(r)),
\end{align*}
where
\begin{align*}
g_n(r)= &\,(G(r,s)f)^p\left(|{\rm Tr}(Q(r)D^2\theta_n)|-\langle b(r,\cdot),\nabla \theta_n\rangle\right)\\
&+2p (G(r,s)f)^{p-1}|\langle Q(r)\nabla\theta_n,\nabla_x G(r,s)f\rangle|+\frac{|x-y|}{t-s} |\nabla \theta_n|(G(r,s)f)^p
\end{align*}
and $h(r)= (G(r,s)f)^{-1}|Q^{1/2}(r)\nabla_x G(r,s)f|$.
Since
\begin{eqnarray*}
\langle b(r,x),\nabla \theta_n(x)\rangle \geq  \eta'\left(\frac{|x|}{n}\right)\frac{C_{s,t}}{n^2}, \qquad\;\, r \in [s,t],
\end{eqnarray*}
where $C_{s,t}$ is the constant in \eqref{inequality-b}, we can estimate
\begin{align}
g_n(r)\leq \frac{C_1}{n^2} \|f\|_{\infty}^p+\frac{1}{n}\Big(2p \Lambda \|f\|_{\infty}^{p-1}\|\,|\nabla f|\,\|_{\infty} +\frac{|x-y|}{t-s} \|f\|^p_{\infty}\Big)=:C(n),
\label{C(n)}
\end{align}
for every $r\in [s,t]$, where $C_1= d\Lambda(2\|\eta'\|_{\infty}+\|\eta''\|_{\infty})+\|\eta'\|_{\infty}C_{s,t}$.

Recalling that $\gamma^2-\beta \gamma\geq -\beta^2/4$ for every $\beta,\gamma \in \R$ and
$G(t,s)g_1\geq G(t,s)g_2$ for every $t\ge s$ if $g_1\geq g_2$ (see Lemma \ref{lemm-prel}) from \eqref{C(n)} we deduce that
\begin{eqnarray*}
\frac{d}{dr}\log \Phi_n(r)\leq \frac{C(n)}{C_{\Phi}}+\frac{p\,|x-y|^2}{4 (p-1)\eta_0(t-s)^2},
\end{eqnarray*}
for every $n>n_0$.
Integrating with respect to $r$ between $s$ and $t$ we get
\begin{eqnarray*}
\log \Phi_n(t)-\log \Phi_n(s) \leq \frac{C(n)}{C_{\Phi}}(t-s)+\frac{p\,|x-y|^2}{4 (p-1)\eta_0(t-s)},\qquad\;\, n>n_0,
\end{eqnarray*}
and \eqref{Harnack} follows letting $n \to +\infty$.

{\em Step 2.}
Let $f \in C_b(\Rd)$ be a nonnegative function; we can consider a sequence  $(f_n)_n\subset C^1_b(\Rd)$ of nonnegative functions
converging to $f$ uniformly on compact sets of $\Rd$ and such that $\|f_n\|_{\infty}\leq \|f\|_{\infty}$.
Then, by Step 1 we have
\begin{equation*}
|(G(t,s)f_n)(x)|^p\leq (G(t,s)|f_n|^p)(y)\exp{\left(\frac{p\,|x-y|^2}{4(p-1)\eta_0(t-s)}\right)},
\end{equation*}
for every $t>s\in I$ and $x,y\in\Rd$. Taking into account formula \eqref{rep_nucleo}, this yields the claim by the dominated convergence theorem.
\end{proof}

The second announced characterization of the supercontractivity of $G(t,s)$ is given
in the following theorem. Its proof is based on Propositions \ref{integrability}, \ref{Har_prop}
and also on the first criterion given in Theorem \ref{super_caract}.

\begin{thm}\label{second_chara}
The following properties are equivalent.
\begin{enumerate}[\rm (i)]
\item
The evolution operator $G(t,s)$ is supercontractive;
\item
the function $\varphi_{\lambda}$ belongs to $L^1(\Rd,\mu_s)$ for every $\lambda>0$ and $s\in I$. Moreover,
\begin{equation}
\sup_{s\in I}\|\varphi_{\lambda}\|_{1,\mu_s}<+\infty,\qquad\;\,\lambda>0.
\label{unif-varphi-lambda}
\end{equation}
\end{enumerate}
\end{thm}
\begin{proof}
``$(i)\Rightarrow (ii)$'' Theorem \ref{super_caract} and Proposition \ref{integrability}  show that, if $G(t,s)$ is supercontractive, then $\varphi_{\lambda}$ belongs
to $L^1(\Rd,\mu_s)$ for every $\lambda>0$ and any $s\in I$ and \eqref{unif-varphi-lambda} holds true.

``$(ii)\Rightarrow (i)$'' Let us assume that \eqref{unif-varphi-lambda} holds true and denote by $M_{\lambda}$ the supremum in the left hand side of
such an inequality.
Integrating the Harnack inequality \eqref{Harnack} with respect $d\mu_t(y)$
and recalling that $\{\mu_t: t \in I\}$
is an evolution system of measures, we get
\begin{align}
\int_{\Rd}|f(y)|^p d\mu_s(y)&= \int_{\Rd}(G(t,s)|f|^p)(y)d\mu_t(y)\nnm\\
&\geq|(G(t,s)f)(x)|^p\int_{\Rd}e^{-\frac{p\,|x-y|^2}{4\eta_0(p-1)(t-s)}}d\mu_t(y)\nnm\\
& \geq |(G(t,s)f)(x)|^p\mu_t(B(0,r))\, e^{-\frac{p(r^2+|x|^2)}{2\eta_0(p-1)(t-s)}},
\label{star}
\end{align}
for every $t>s,\,r>0, \,x,y \in \Rd,$ and $f \in C_b(\Rd)$. Hence,
\begin{equation}
|(G(t,s)f)(x)| \leq 2\exp\Big(\frac{R^2+|x|^2}{2\eta_0(p-1)(t-s)}\Big)\|f\|_{p,\mu_s},\qquad\;\, t>s,\;\,x\in \Rd,
\label{form-1}
\end{equation}
where $R$ is such that $\mu_t(B(0,R))>2^{-p}$, for any $t\in I$. Let us now fix $q>p$ and set $\lambda_0=(2\eta_0(p-1)(t-s))^{-1}q$.
By \eqref{form-1} we can estimate
\begin{align}
\int_{\Rd}|(G(t,s)f)(x)|^qd\mu_t(x) &\le  2^q\exp\left (\frac{R^2}{2\eta_0(p-1)(t-s)}\right )\|\varphi_{\lambda_0}\|_{1,\mu_s}\|f\|_{p,\mu_s}^q\notag\\
&\le 2^q\exp\left (\frac{R^2}{2\eta_0(p-1)(t-s)}\right )M_{(2\eta_0(p-1)(t-s))^{-1}q}\|f\|_{p,\mu_s}^q\notag\\
&=: C_{p,q}(t-s)\|f\|_{p,\mu_s}^q,
\label{Cpq}
\end{align}
for any $I\ni s<t$. Now, it is clear the monotonicity of the function $r\mapsto C_{p,q}(r)$ and that, by density, we can extend the previous inequality to any $f\in L^p(\Rd,\mu_s)$. This completes the proof.
\end{proof}

Our aim is now to provide a sufficient condition for the supercontractivity of the evolution operator $G(t,s)$.
First we prove a preliminary lemma.

\begin{lemm}\label{Lya_prop}
Assume that there exist $K,\beta>0$ and $R>1$ such that
\begin{eqnarray*}
\langle b(t,x),x\rangle \leq -K |x|^2 (\log |x|^2)^{\beta},\qquad\;\, t\in I,\;\, |x|\ge R.
\end{eqnarray*}
Then, any positive $\psi_{\lambda,\delta}\in C^2(\Rd)$ which agrees with the function $x\mapsto e^{\lambda |x|^2(\log|x|^2)^{\delta}}$ for any $x\in \Rd\setminus B(0,R)$, is a Lyapunov function
 satisfying \eqref{Lyapunov} for every $\lambda <K(2\Lambda)^{-1}$, if $\delta=\beta$, and for every $\lambda>0$, if $\delta\in [0,\beta)$.
\end{lemm}
\begin{proof}
A straightforward computation shows that
\begin{align*}
(\mathcal{A}(t)\psi_{\lambda,\delta})(x)=2\lambda\psi_{\lambda,\delta}(x)\bigg\{&2\lambda (\log |x|^2)^{2\delta}\langle Q(t)x,x\rangle
+2\delta^2\lambda(\log |x|^2)^{2\delta-2}\langle Q(t)x,x\rangle\\
&+ 4\delta\lambda(\log |x|^2)^{2\delta-1}\langle Q(t)x,x\rangle
+{\rm Tr}(Q(t))(\log |x|^2)^{\delta}\\
&+2\delta(\log |x|^2)^{\delta-1}\frac{\langle Q(t)x,x\rangle}{|x|^2}+\delta {\rm Tr}(Q(t))
(\log |x|^2)^{\delta-1}\\
&+2\delta(\delta-1)\lambda (\log |x|^2)^{\delta-2}\frac{\langle Q(t)x,x\rangle}{|x|^2}\\
&+\langle b(t,x),x\rangle (\log |x|^2)^{\delta}+\delta\langle b(t,x),x\rangle (\log |x|^2 )^{\delta-1}\bigg\}\\
\le 2\lambda\psi_{\lambda,\delta}(x)\bigg\{&2\lambda\Lambda |x|^2(\log |x|^2)^{2\delta}-K|x|^2(\log |x|^2)^{\beta+\delta}\\
&+o(|x|^2(\log |x|^2)^{\beta+\delta})\bigg\},
\end{align*}
as $|x|\to +\infty$. Hence, the function in brackets tends to $-\infty$ as $|x|\to +\infty$, if $\gamma$ and $\lambda$ are as in the statement
of the lemma. It is now immediate to show that there exist two positive constants $a=a(\lambda,\delta)$ and $\gamma=\gamma(\lambda,\delta)$ such that
${\mathcal A}(t)\psi_{\lambda,\delta}\le a-\gamma \psi_{\lambda,\delta}$ for any $t\in I$ and \eqref{Lyapunov} holds.
\end{proof}

\begin{thm}\label{suff_cond_super}
Assume that there exist $K_1>0$ and $R>1$ such that
\begin{equation}\label{super_ipo-1}
\langle b(t,x),x\rangle \leq K_1 |x|^2 \log |x|,\qquad\;\, t\in I,\;\, |x|\ge R.
\end{equation}
Then, the evolution operator $G(t,s)$ is supercontractive.
\end{thm}
\begin{proof}
In view of \cite[Thm. 5.4]{KunLorLun09Non}, the proof is an immediate consequence of Theorem \ref{second_chara}
and Lemma \ref{Lya_prop}.
\end{proof}

\begin{rmk}\rm{
The condition \eqref{super_ipo-1} is quite optimal. Indeed, the autonomous operator
$(\mathcal{A}\zeta)(x)=\Delta \zeta(x)-\langle x,\nabla \zeta(x)\rangle$ does not satisfy it
and, in fact, it is well known that the associated
Ornstein-Uhlenbeck semigroup is not supercontractive with respect to the Gaussian invariant
measure $d\mu(x)=(2\pi)^{-d/2}e^{-|x|^2/2}dx$, as proved in \cite{Nelson}.}
\end{rmk}

\section{Ultraboundedness}\label{section_ultrabdd}

In this section we provide a condition equivalent to the ultraboundedness
property of the evolution operator $G(t,s)$. As in \cite{DaPLun10Ult,RocWan03Sup,wang_97,wan04Fun}), which deal with the autonomous case, 
we use the Harnack type estimate \eqref{Harnack}
satisfied by $G(t,s)$ to get ultraboundedness of $G(t,s)$.
However, we need to strengthen assumption \eqref{Lyapunov}, as next theorem shows.

\begin{thm}\label{ultrathm}
Assume that, for any $\lambda>0$, there exist $R=R(\lambda)>0$,
a convex increasing function $h_{\lambda}:[0,+\infty)\to \R$ such that
$1/h_{\lambda}\in L^1(c,+\infty)$ for large $c=c(\lambda)$ and
\begin{equation}
(\mathcal{A}(t)\varphi_{\lambda})(x)\leq -h_{\lambda}(\varphi_{\lambda}(x)),\qquad\;\, t\in I,\;\, |x|\geq R,
\label{conv}
\end{equation}
where $\varphi_\lambda$ is the Gaussian function defined by $\varphi_\lambda(x):=e^{\lambda|x|^2}$ for any $x \in \Rd$.
Then, $G(t,s)$ is
ultrabounded and it maps $L^p(\Rd, \mu_s)$ into $C_b(\Rd)$ for every $p>1$.
\end{thm}

\begin{proof}
We prove the claim for $p\in (1,2]$. For $p>2$, estimate \eqref{est_ultrabdd} will follow from the H\"older inequality. We split
the proof into two steps. First, we consider the case $p=2$ and, then, the case $p\in (1,2)$.

{\em Step 1.} An insight into the proof of \cite[Thm. 3.3]{Lun10Com} (see also \cite[Thm. 4.3]{AngLor10Com} for further details) shows that, under our assumptions, the function $t \mapsto (G(t,s)\varphi_\lambda)(x)$ is well defined for each $t>s$ and $x\in \Rd$, and
$G(t,s)\varphi_\lambda\in L^{\infty}(\Rd)$
for every $t>s$ and $\lambda>0$. More precisely, if for every $\delta,\lambda>0$ we set
\begin{equation}\label{mdelta}
M_{\delta,\lambda}:= \sup_{{x\in \Rd }\atop{t-s\geq \delta}}(G(t,s)\varphi_\lambda)(x),
\end{equation}
then $M_{\delta,\lambda}$ turns out to be a positive constant
 independent of $t$ and $s$.
This is enough to establish \eqref{est_ultrabdd} with $p=2$. Indeed, integrating both sides
of estimate \eqref{Harnack} (with $p=2$) with respect to $d\mu_t(y)$ and arguing as in the proof of Theorem \ref{second_chara},
we get
\begin{eqnarray*}
|(G(t,s)f)(x)| \leq 2\exp\left (\frac{R+|x|^2}{2\eta_0(t-s)}\right)\|f\|_{2,\mu_s},\qquad\;\, t>s,\;\,x\in \Rd,
\end{eqnarray*}
where $R$ is such that $\mu_t(B(0, R))>\frac{1}{4}$. Hence we obtain
\begin{align}\label{interm}
\|G(t,s)f\|_{\infty}&= \|G\left(t,(t+s)/2\right)G\left((t+s)/2,s\right)f\|_{\infty}\nnm\\
&\leq 2 e^{\frac{R}{2\eta_0(t-s)}}\|f\|_{2,\mu_s}\|G\left(t,(t+s)/2\right)\varphi_{\lambda_0}\|_{\infty},
\end{align}
for every $f \in C_b(\Rd)$ and for $\lambda_0= \frac{1}{2\eta_0(t-s)}$. Formulas \eqref{mdelta} and \eqref{interm} yield
\begin{equation}\label{smooth}
\|G(t,s)f\|_{\infty} \leq C_{2,\infty}(t-s)\|f\|_{2,\mu_s},\qquad\;\, t>s,\;\, f \in C_b(\Rd),
\end{equation}
with
\begin{eqnarray}\label{est_2_pippo}
C_{2,\infty}(t-s)=2 e^{\frac{R}{2\eta_0(t-s)}}M_{\frac{t-s}{2},\,\frac{1}{2\eta_0(t-s)}}.
\end{eqnarray}
The monotonicity of the function $r\mapsto C_{2,\infty}(r)$ is immediate consequence of the fact that
$M_{\delta_2,\lambda_1}\leq M_{\delta_1,\lambda_2}$,
for every $0<\delta_1\le \delta_2$ and $0<\lambda_1\le\lambda_2$, as it can be easily proved.

Now, let $f \in L^2(\Rd, \mu_s)$ and consider $f_n\in C_b(\Rd)$ converging to $f$ in $L^2(\Rd, \mu_s)$ as $n \to +\infty$.
Since $G(t,s)$ is a contraction from $L^2(\Rd, \mu_s)$ to $L^2(\Rd, \mu_t)$, $G(t,s)f_n$
converges to $G(t,s)f$ in $L^2(\Rd, \mu_t)$ as $n \to +\infty$.
Moreover
\begin{equation}\label{unifor}
\|G(t,s)f_n-G(t,s)f_m\|_{\infty}\leq C_{2,\infty}(t-s)\|f_n-f_m\|_{2,\mu_s},
\end{equation}
for every $t>s$, and $n, m \in \N$. Formula \eqref{unifor} yields that the sequence $G(t,s)f_n$ converges uniformly
in $\Rd$ to some function $g\in C_b(\Rd)$ and that
$g=G(t,s)f$. Then, we conclude writing \eqref{smooth} for $f_n$ and letting $n \to +\infty$.

{\em Step 2.} To prove \eqref{est_ultrabdd} when $p\in (1,2)$, we observe that
\begin{eqnarray*}
\int_{\Rd}\varphi_{\lambda,n}d\mu_s(x) =\int_{\Rd}G(s+1,s)\varphi_{\lambda,n}d\mu_{s+1}(x),
\end{eqnarray*}
for any $\lambda>0$, any $s\in I$ and any $n\in\N$, where $\varphi_{\lambda,n}=\min\{\varphi_{\lambda},n\}$.
Letting $n\to +\infty$ and using \eqref{mdelta} with $\delta=1$, we obtain
\begin{eqnarray*}
\int_{\Rd}\varphi_{\lambda}d\mu_s(x) =\int_{\Rd}G(s+1,s)\varphi_{\lambda}\,d\mu_{s+1}(x)\le M_{1,\lambda},\qquad\;\,s\in I.
\end{eqnarray*}
Hence, condition \eqref{unif-varphi-lambda} is satisfied, and Theorem \ref{second_chara} shows that
the evolution operator $G(t,s)$ is supercontractive. Therefore,
\begin{align*}
\|G(t,s)f\|_{\infty}=&\|G(t,(t+s)/2)G((t+s)/2,s)f\|_{\infty}\\
\le &\|G(t,(t+s)/2)\|_{2\to\infty}\|G((t+s)/2,s)f\|_{2,\mu_{(t+s)/2}}\\
\le & C_{2,\infty}((t-s)/2)C_{p,2}((t-s)/2)\|f\|_{p,\mu_s},
\end{align*}
for any $f\in L^p(\Rd,\mu_s)$ and any $s,t\in I$ with $s<t$.
This completes the proof.
\end{proof}

\begin{rmk}
{\rm Each function $\varphi_{\lambda}$, as in Theorem \ref{ultrathm}, satisfies Hypothesis \ref{hyp1}(iii), i.e., it is a Lyapunov function for the
nonautonomous elliptic operators ${\mathcal A}(t)$.
Indeed, since $h_{\lambda}$ is a convex function which tends to $+\infty$ as $r\to +\infty$, there exist $a_{\lambda}>0$ and $b_{\lambda}\in\R$ such that $h_{\lambda}(r)\ge a_{\lambda}r+b_{\lambda}$ for any $r\ge 0$.
From \eqref{conv} it thus follows that $(\mathcal{A}(t)\varphi_{\lambda})(x)\le -a_{\lambda}\varphi_{\lambda}(x)+b_{\lambda}$
for any $t\in I$ and any $x\in \Rd\setminus B(0,R)$. Up to replacing $b_{\lambda}$ with a larger constant, if
needed, we can assume that the previous inequality is satisfied by any $x\in\Rd$, so that \eqref{Lyapunov} is satisfied.

From \cite[Thm. 5.4]{KunLorLun09Non}, we deduce that $\sup_{s\in I}\|\varphi_{\lambda}\|_{1,\mu_s}<+\infty$ for any $\lambda>0$, and this gives an alternative proof of the first part of Step 2 in Theorem \ref{ultrathm}.
}
\end{rmk}

As a consequence of Theorem \ref{ultrathm} we now provide a sufficient condition for $G(t,s)$ to be ultrabounded.

\begin{thm}\label{ipo_ultrbdd}
Suppose that there exist three positive constants $K_2$, $\alpha>1$ and $R_0>1$
such that
\begin{eqnarray}\label{ultrbdd_ipo}
\langle b(t,x),x\rangle\le -K_2|x|^{2}(\log|x|)^{\alpha},\qquad\;\, t\in I,\;\,|x|\ge R_0.
\end{eqnarray}
Then, $G(t,s)$ is ultrabounded.
\end{thm}

\begin{proof}
A straightforward computation shows that
\begin{align*}
(\mathcal{A}(t)\varphi_{\lambda})(x)&=  2\lambda\varphi_{\lambda}(x)\left [\textrm{Tr}(Q(t))+2\lambda\langle Q(t)x,x\rangle
+\langle b(t,x),x\rangle\right ]\notag\\
&\le  -2\lambda\varphi_{\lambda}(x)\left [K_2|x|^{2}(\log|x|)^{\alpha}-2\lambda\Lambda |x|^2-\Lambda d\right ],
\end{align*}
for any $t\in I$ and any $x\in\Rd\setminus B(0,R_0)$.
Let now $C_{\alpha}$ be a positive constant such that
\begin{eqnarray*}
2\lambda\Lambda y^2\le \frac{K_2}{2}y^{2}(\log y)^\alpha+C_{\alpha},\qquad\;\,y\ge R_0.
\end{eqnarray*}
Then,
\begin{align*}
(\mathcal{A}(t)\varphi_{\lambda})(x)
\le  -\lambda\varphi_{\lambda}(x)\left [K_2|x|^{2}(\log |x|)^{\alpha}-2C_{\alpha}-2\Lambda d\right ]=-g_{\lambda}(\varphi_{\lambda}(x)),
\end{align*}
for any $t\in I$ and any $x\in\Rd\setminus B(0,R_0)$. Here,
\begin{eqnarray*}
g_{\lambda}(y)=y\left [K_2 2^{-\alpha}\log y\big(\log(\lambda^{-1}\log y))^{\alpha}-2\lambda C_{\alpha}-2\lambda\Lambda d\right ],\qquad\;\,y\ge e^{\lambda}.
\end{eqnarray*}
$g_{\lambda}$ is a convex function in the interval $[e^{\lambda},+\infty)$ and, since $g_\lambda(y)\sim y\log y(\log(\log y))^\alpha$ as $y\to+\infty$, $1/g_{\lambda}$ is integrable in a neighborhood of $+\infty$.
On the other hand, $g_{\lambda}$ is not increasing in $[e^{\lambda},+\infty)$ since $g'_{\lambda}(e^{\lambda})=-2(C_{\alpha}+\Lambda d)$.
To overcome this difficulty, let us introduce the function $h_{\lambda}=g_{\lambda}(y_{0,\lambda})\chi_{[0,y_{0,\lambda}]}+g_{\lambda}\chi_{(y_{0,\lambda},+\infty)}$, where $y_{0,\lambda}>e^{\lambda}$ is the point where the minimum of the function $g_{\lambda}$ is attained. Clearly, $h_{\lambda}$ is a convex and increasing function in $[0,+\infty)$
which equals $g_{\lambda}$ in $[y_{0,\lambda},+\infty)$. Moreover, $h_{\lambda}\le g_{\lambda}$ in $[e^{\lambda},+\infty)$, therefore
$(\mathcal{A}(t)\varphi_{\lambda})(x)\le -h_{\lambda}(\varphi_{\lambda}(x))$ for any $t\in I$ and any $|x|\ge R_0$. We can thus apply
Theorem \ref{ultrathm}.
\end{proof}

\begin{rmk}{\rm
The condition \eqref{ultrbdd_ipo} is rather sharp. Indeed in \cite{KavKerRoy93Que}, the authors consider the autonomous operator
$(\mathcal{A}\zeta)(x)=\Delta \zeta(x)-\langle \nabla\Phi(x),\nabla \zeta(x)\rangle$,
where $\Phi$ is such that $e^{-\Phi}\in L^1(\Rd)$, and prove that, if $\Phi(x)\sim |x|^2\log|x|$ as $|x|\to +\infty$, then the semigroup $T(t)$ associated to $\mathcal{A}$ in $C_b(\Rd)$ is not ultrabounded in the Lebesgue spaces with respect the invariant measure
$d\mu(x)=\|e^{-\Phi}\|^{-1}_{1}e^{-\Phi(x)}dx$.
}\end{rmk}

The Harnack type estimate
\eqref{Harnack} and the fact that $G(t,s)\varphi_\lambda \in L^\infty(\Rd)$ for every $\lambda>0$ and $t>s$ represent the
 key tools used in the proof of Theorem \ref{ultrathm} to get ultraboundedness.
Hypotheses \ref{hyp1} are enough to prove the Harnack formula \eqref{Harnack}. On the other hand, to prove
that $G(t,s)\varphi_\lambda \in L^\infty(\Rd)$ for every $\lambda>0$ and $t>s$ we have strengthened our assumptions requiring
the additional condition \eqref{conv}.
The condition $G(t,s)\varphi_\lambda \in L^\infty(\Rd)$ for every $\lambda>0$ and $t>s$ is optimal to get
ultraboundedness of $G(t,s)$ for every $t>s$.
The proof of this fact is based on the occurrence of the family of
logarithmic Sobolev inequalities \eqref{super_LSI}
and the consequent measure concentration result proved
in Proposition \ref{integrability}.

\begin{thm}
The evolution operator $G(t,s)$ is ultrabounded if and only if, for every $\lambda>0$ and $t>s$,
the function $G(t,s)\varphi_\lambda$ belongs to $L^\infty(\Rd)$ and, for any $\delta,\lambda>0$, there exists a positive
constant $K_{\delta,\lambda}$ such that
\begin{equation}
\|G(t,s)\varphi_{\lambda}\|_{\infty}\le K_{\delta,\lambda},\qquad\;\,s,t\in I,\;\, t-s\ge\delta.
\label{italia-qualificata}
\end{equation}
\end{thm}
\begin{proof}
In view of the proof of Theorem \ref{ultrathm} the ``if'' part of the statement is true.

Conversely, if $G(t,s)$ is ultrabounded, then it is bounded from $L^p(\Rd,\mu_s)$ into $L^q(\Rd,\mu_t)$ for every $t>s$ and $1< p< q<+\infty$,
and
\begin{equation*}
\|G(t,s)\|_{p\to q}\leq\|G(t,s)\|_{p\to\infty}<+\infty.
\end{equation*}
By Proposition \ref{super}, the logarithmic Sobolev inequality \eqref{fam_log_sob_prop} holds.
Consequently, from Proposition \ref{integrability} we deduce
that $\varphi_\lambda \in L^1(\Rd,\mu_s)$ for every $\lambda>0$ and $s\in I$, and
$\sup_{s\in I}\|\varphi_{\lambda}\|_{1,\mu_s}<+\infty$.
Therefore,
\begin{eqnarray*}
\|G(t,s)\varphi_{\lambda}\|_{\infty}\leq
\|G(t,s)\|_{2\to\infty}\|\varphi_\lambda\|_{2,\mu_s}=
\|G(t,s)\|_{2\to\infty}\|\varphi_{2\lambda}\|_{1,\mu_s}^{\frac{1}{2}}<+\infty,
\end{eqnarray*}
for any $t>s$ and $\lambda>0$.

Now, fix $\delta>0$ and let $t-s\ge\delta$. Since the function $r\mapsto C_{2,\infty}(r)$ is decreasing, we get
\eqref{italia-qualificata} with $K_{\delta,\lambda}=C_{2,\infty}(\delta)\sup_{s\in I}\|\varphi_{2\lambda}\|_{1,\mu_s}^{1/2}$.
\end{proof}

\section{Ultracontractivity}
\label{section_ultracontra}
In this section we assume the following additional assumption on the drift term of  the operators $\mathcal{A}(t)$.
\begin{hyp}
\label{hyp-2}
There exist three positive constants  $K_{3}$, $R$ and $\kappa >2$ such that
\begin{equation}\label{ex_1_1}
\langle b(t,x),x\rangle \leq - K_{3}|x|^{\kappa},\qquad\;\,t\in I,\;\,x\in\Rd\setminus B(0,R).
\end{equation}
\end{hyp}

\subsection{$L^1$-$L^2$ integrability}
To begin with, let us give an estimate of the asymptotic behaviour of the function $\beta$ defined in \eqref{fam_log_sob_prop} near zero.

\begin{prop}
\label{coro-5.2}
Under Hypotheses $\ref{hyp1}$ and $\ref{hyp-2}$,
$\beta(\varepsilon)=O(\varepsilon^{-\frac{\kappa}{\kappa-2}})$ as $\varepsilon\to 0^+$.
\end{prop}

\begin{proof}
First of all, let us prove that the function $x\mapsto \varphi_{\delta,\kappa}(x)=e^{\delta |x|^{\kappa}}$ belongs to $L^1(\Rd,\mu_s)$ for any $s\in I$ and any $\delta<K_3/(\kappa\Lambda)$ (see
\eqref{ell}) and that there exists a positive constant $M$, independent of $s$, such that
$\|\varphi_{\delta,\kappa}\|_{1,\mu_s}\le M$ for any $s\in I$.
For this purpose, in view of \cite[Thm. 5.4]{KunLorLun09Non} we can limit ourselves to proving that $(\mathcal{A}(t)\varphi_{\delta,\kappa})(x)\le a_1-\gamma_1\varphi_{\delta,\kappa}(x)$ for any $t\in I$, $x\in\Rd$ and
some positive constants $a_1$ and $\gamma_1$.
It is easy to compute and to estimate $\mathcal{A}(t)\varphi_{\delta,\kappa}$
in the following way:
\begin{align*}
(\mathcal{A}(t)\varphi_{\delta,\kappa})(x)=&  \delta \kappa \varphi_{\delta,\kappa}(x) \big[ (\delta \kappa |x|^{2\kappa -4} + (\kappa-2) |x|^{\kappa -4} )\langle Q(t)x,x\rangle\\
&\qquad\qquad\;\; + \textrm{Tr}(Q(t))|x|^{\kappa - 2}  + \langle b(t,x),x\rangle |x|^{\kappa - 2}\big]\\
\leq &  \delta \kappa \varphi_{\delta,\kappa}(x) [\delta \kappa \Lambda |x|^{2\kappa -2} + \Lambda (d+  \kappa -2) |x|^{\kappa -2} -  K_{3}|x|^{2\kappa -2}]\\
=: & g_1(x) \varphi_{\delta,\kappa}(x),
\end{align*}
for any $(t,x)\in I\times\Rd$, where $g_1(x)$ tends to $-\infty$ as $|x|\to +\infty$. Hence, the claim follows at once.

We now observe that, for any $\lambda>0$ and any $t\ge 0$, we have
\begin{eqnarray*}
\delta t^{\kappa}-\lambda t^2\ge\left (\frac{2}{\kappa\delta}\right )^{\frac{2}{\kappa-2}}\frac{2-\kappa}{\kappa}\lambda^{\frac{\kappa}{\kappa-2}}=:-c_1\lambda^{\frac{\kappa}{\kappa-2}}.
\end{eqnarray*}
It thus follows that
\begin{equation}
\|\varphi_{\lambda}\|_{1,\mu_s}\le e^{c_1\lambda ^{\frac{\kappa}{\kappa-2}}}\|\varphi_{\delta,\kappa}\|_{1,\mu_s}\le e^{c_1\lambda ^{\frac{\kappa}{\kappa-2}}}\sup_{r\in I}\|\varphi_{\delta,\kappa}\|_{1,\mu_r}=:
c_2e^{c_1\lambda ^{\frac{\kappa}{\kappa-2}}}.
\label{cond-int-c1-c2}
\end{equation}

Writing \eqref{super_LSI} with $p=2$ and $q=3$,  we get
\begin{align*}
\int_{\Rd}f^2 \log \left(\frac{|f|}{\|f\|_{2,\mu_s}} \right)d\mu_s(x) \leq &(1-e^{2r_0(t-s)})\frac{8\Lambda }{|r_0|}\|\,|\nabla f|\,\|_{2,\mu_s}^2\nnm\\
&+ 3\log(\tilde C_{2,3}(t,s))\|f\|_{2,\mu_s}^2,
\end{align*}
for any $I\ni s<t$. Let us provide an estimate of the constant $\tilde C_{2,3}(t,s)$, which represents the norm of $G(t,s)$ from
$L^2(\Rd,\mu_s)$ to $L^3(\Rd,\mu_t)$.
From \eqref{Cpq}, with $p=2$ and $q=3$, we get
\begin{eqnarray*}
\tilde C_{2,3}(t,s)\le C_{2,3}(t-s)= 2\exp\left (\frac{R^2}{6\eta_0(t-s)}\right )\|\varphi_{\lambda_0}\|_{1,\mu_s}^{\frac{1}{3}},
\end{eqnarray*}
where
$\lambda_0=3(2\eta_0(t-s))^{-1}$. From estimate \eqref{cond-int-c1-c2} we thus conclude that
\begin{eqnarray*}
C_{2,3}(t-s)\le c_3\exp\left (\frac{R^2}{6\eta_0(t-s)}\right )\exp\left (\frac{c_4}{(t-s)^{\frac{\kappa}{\kappa-2}}}\right ),
\end{eqnarray*}
for some positive constants $c_3$ and $c_4$, so that
\begin{eqnarray*}
\log(C_{2,3}(t-s))\le \log(c_3)+c_4(t-s)^{-\frac{\kappa}{\kappa-2}}+c_5(t-s)^{-1}.
\end{eqnarray*}
Now, we fix $\varepsilon<8\Lambda|r_0|^{-1}$ and solve the equation $8\Lambda |r_0|^{-1}(1-e^{2r_0(t-s)})=\varepsilon$. We get
\begin{eqnarray*}
t-s=\frac{1}{2r_0}\log\left (1+\frac{r_0}{8\Lambda}\varepsilon\right ).
\end{eqnarray*}
Hence, for $\varepsilon<8\Lambda|r_0|^{-1}$, we obtain
\begin{eqnarray*}
\beta(\varepsilon)\le 3\left\{\log(c_3)+c_4\left [\frac{1}{2r_0}\log\left (1+\frac{r_0}{8\Lambda}\varepsilon\right )\right ]^{-\frac{\kappa}{\kappa-2}}
+c_5\left [\frac{1}{2r_0}\log\left (1+\frac{r_0}{8\Lambda}\varepsilon\right )\right ]^{-1}\right\},
\end{eqnarray*}
and the assertion follows at once.
\end{proof}

We can now prove the boundedness of $G(t,s)$ from $L^1(\Rd,\mu_s)$ into $L^2(\Rd,\mu_t)$ following the basic ideas in the proof of \cite[Thm. 3.4]{maheux} for the autonomous case. 
We stress that the nonautonomous setting gives rise to some additional technical difficulties.

\begin{thm}\label{teo_12}
Under Hypotheses $\ref{hyp1}$ and $\ref{hyp-2}$, for any $s,t\in I$, with $s<t$, the operator $G(t,s)$ is bounded from $L^1(\Rd,\mu_s)$ into $L^2(\Rd,\mu_t)$.
\end{thm}

\begin{proof}

As a first step we observe that, for any $s\in I$ and any nonnegative $g\in C_b(\Rd)$,
\begin{equation}
2\|g\|_{2,\mu_s}^2\log\|g\|_{2,\mu_s}-\|g\|_{2,\mu_s}^2\log\|g\|_{1,\mu_s}\le\int_{\Rd}g^2\log g\, d\mu_s(x).
\label{LS-2}
\end{equation}
It suffices to prove \eqref{LS-2} for functions with $\|g\|_{1,\mu_s}=1$, which reduces to
\begin{equation}
2\|g\|_{2,\mu_s}^2\log\|g\|_{2,\mu_s}\le\int_{\Rd}g^2\log g\, d\mu_s(x),
\label{LS-1}
\end{equation}
since \eqref{LS-2} in the general case will follow from applying \eqref{LS-1} to the function $\|g\|_{1,\mu_s}^{-1}g$.

To prove estimate \eqref{LS-1} we observe that the measure $d\nu_s(x)=gd\mu_s(x) $ is a probability measure and the function
$\psi(x)=x\log x$ is convex in $(0,+\infty)$. Therefore, Jensen inequality yields
\begin{align*}
\psi\left (\int_{\Rd}g d\nu_s(x)\right )\le\int_{\Rd}\psi(g)d\nu_s(x),
\end{align*}
which is \eqref{LS-1}.

We now fix a positive function $f\in C_c^{\infty}(\Rd)$, with $\|f\|_{1,\mu_s}=1$. Applying the logarithmic Sobolev inequality
\eqref{fam_log_sob_prop} with $f$ and $\mu_s$ being replaced respectively by $\theta_nG(r,s)f$ and $\mu_r$
(where $\theta_n$ is defined in \eqref{thetan}) and taking \eqref{LS-2} (with $g=\theta_nG(r,s)f$) into account,
we obtain
\begin{align}
\|\theta_n G(r,s)f\|_{2,\mu_r}^2\log\left (\frac{\|\theta_n G(r,s)f\|_{2,\mu_r}}{\|\theta_n G(r,s)f\|_{1,\mu_r}}\right )
\le & \varepsilon \|\,|\nabla_x (\theta_n G(r,s)f)|\,\|_{2,\mu_r}^2\nnm\\
&+\beta(\varepsilon)\|\theta_n G(r,s)f\|_{2,\mu_r}^2,
\label{pippo}
\end{align}
for every $s,r\in I$ with $s\le r$.
Since $\{\mu_r:\,r\in I\}$ is a tight evolution system of measures, we can fix $R\in\N$ such that $\mu_r(B(0,R))\ge 1/2$ for every $r\in I$.
Now, let us fix $t>s$ and set
\begin{eqnarray*}
\zeta_n(r):=\log(\|\theta_nG(r,s)f\|_{2,\mu_r}^2), \qquad\;\, n\ge R,\; \,r\in [s,t].
\end{eqnarray*}
Note that the function $\zeta_n$ is well defined since $G(r,s)f$ is a smooth and everywhere positive function in $\Rd$ for any $r\ge s$, (see Lemma \ref{lemm-prel}). Hence,
$\|\theta_n G(r,s)f\|_{2,\mu_r}\ge \delta/\sqrt{2}$ for any $r\in [s,t]$ and $n\ge R$, where $\delta$ denotes the minimum of the function $G(\cdot,s)f$ in
$[s,t]\times B(0,R)$. From Proposition \ref{lemma-3.1} we deduce that the function $\zeta_n$ is differentiable in $[s,t]$ and
\begin{align*}
\|\theta_nG(r,s)f\|_{2,\mu_r}^{2}\zeta_n'(r)=&\,2\int_{\Rd}\theta_n^2(G(r,s)f)\mathcal{A}(r)G(r,s)fd\mu_r(x) \\
&-\int_{\Rd}{\mathcal A}(r)[\theta_n^2(G(r,s)f)^2]d\mu_r(x) \\
=& -2\int_{\Rd}\langle Q(r)\nabla_x(\theta_n G(r,s)f),\nabla_x(\theta_n G(r,s)f)\rangle d\mu_r(x) \\
&-4\int_{\Rd}\theta_n(G(r,s)f)\langle Q(r)\nabla \theta_n,\nabla_xG(r,s)f\rangle d\mu_r(x) \\
&-2\int_{\Rd}\theta_n(G(r,s)f)^2\mathcal{A}(r)\theta_nd\mu_r(x) .
\end{align*}
Using \eqref{ex_1_1} we can estimate
\begin{eqnarray*}
-\mathcal{A}(r)\theta_n\le \frac{d\Lambda}{n^2}(2\|\eta'\|_{\infty}+\|\eta''\|_{\infty})-\eta'\left(\frac{|x|}{n}\right)\frac{\langle b(r,x),x\rangle}{n|x|}
\le \frac{d\Lambda}{n^2}(2\|\eta'\|_{\infty}+\|\eta''\|_{\infty}).
\end{eqnarray*}
It thus follows that
\begin{align*}
\|\theta_nG(r,s)f\|_{2,\mu_r}^{2}\zeta_n'(r)
\le & -2\eta_0\|\,|\nabla_x(\theta_n G(r,s)f)|\,\|_{2,\mu_r}^2+\frac{4\Lambda}{n}\|\eta'\|_{\infty}\|f\|_{\infty}\|\,|\nabla f|\,\|_{\infty}\\
&+\frac{2d\Lambda}{n^2} (2\|\eta'\|_{\infty}+\|\eta''\|_{\infty})\|f\|^2_{\infty},
\end{align*}
for any $r\ge s$ and $n\in\N$.
Hence, from \eqref{pippo} and observing that $\|\theta_nG(r,s)f\|_{1,\mu_r}\le \|G(r,s)f\|_{1,\mu_r}\le\|f\|_{1,\mu_s}\le 1$,  we deduce that
\begin{align}
\zeta_n'(r)\leq & -\frac{\eta_0}{\varepsilon}\zeta_n(r)+\frac{2\eta_0\beta(\varepsilon)}{\varepsilon}+\frac{2C(n)}{\delta^2},\qquad\;\, r\in [s,t],
\label{zetaaa}
\end{align}
where
\begin{align*}
C(n)=\frac{4\Lambda}{n}\|\eta'\|_{\infty}\|f\|_{\infty}\|\,|\nabla f|\,\|_{\infty}+\frac{2d\Lambda}{n^2}(2\|\eta'\|_{\infty}+\|\eta''\|_{\infty})\|f\|_{\infty}^2.
\end{align*}
Fix  $m>2/(\kappa-2)$ and take $\varepsilon=\eta_0 (r-s)/(m+1)$ in the previous inequality. Multiplying both of the sides of \eqref{zetaaa} by $(r-s)^{m+1}$ and integrating between $s$ and $t$ we get
\begin{align*}
\int_s^t(r-s)^m\zeta_n(r)dr\le &-\frac{1}{m+1}\int_s^t(r-s)^{m+1}\zeta'_n(r)dr\\
&+2\int_s^t(r-s)^m\beta\left (\frac{\eta_0(r-s)}{m+1}\right )dr+\frac{2C(n)(t-s)^{m+2}}{\delta^2(m+1)(m+2)}.
\end{align*}
Note that the last integral term in the right-hand side of the previous inequality is finite due to
Proposition \ref{coro-5.2}. An integration by parts shows that
\begin{eqnarray*}
\frac{1}{m+1}\int_s^t(r-s)^{m+1}\zeta_n'(r)dr=
\frac{1}{m+1}(t-s)^{m+1}\zeta_n(t)-\int_s^t (r-s)^m\zeta_n(r)dr.
\end{eqnarray*}
Hence,
\begin{align*}
(t-s)^{m+1}\zeta_n(t)\le & 2(m+1)\int_s^t(r-s)^m\beta\left (\frac{\eta_0(r-s)}{m+1}\right )dr+\frac{2C(n)}{\delta^2(m+2)}(t-s)^{m+2}\\
\le &\frac{2(m+1)^{m+2}}{\eta_0^{m+1}}\int_0^{\frac{\eta_0(t-s)}{m+1}}\sigma^m\beta(\sigma)d\sigma+ \frac{2C(n)}{\delta^2}(t-s)^{m+1}.
\end{align*}
Then, letting $n \to +\infty$ it follows that
\begin{align*}
(t-s)^{m+1}\log(\|G(t,s)f\|_{2,\mu_t}^2)&\le\frac{2(m+1)^{m+2}}{\eta_0^{m+1}}\int_0^{\frac{\eta_0(t-s)}{m+1}}\sigma^m\beta(\sigma)d\sigma\nnm\\
&\leq C(\kappa,\eta_0)(t-s)^{m+1-\frac{\kappa}{\kappa-2}},
\end{align*}
for some positive constant $C(\kappa,\eta_0)$. Thus we get
\begin{equation}\label{formula}
\|G(t,s)f\|_{2,\mu_t}\le e^{\frac{C(\kappa,\eta_0)}{2(t-s)^{\kappa/(\kappa-2)}}}= e^{\frac{C(\kappa,\eta_0)}{2(t-s)^{\kappa/(\kappa-2)}}}\|f\|_{1,\mu_s}.
\end{equation}
By homogeneity, we can extend \eqref{formula} to any positive and smooth function $f$ with $\|f\|_{1,\mu_s}\neq 1$.
Next, for a general $f\in C^{\infty}_c(\Rd)$, we write \eqref{formula} for $f_n=(f^2+n^{-1})^{1/2}$.
Observing that $\|G(t,s)f_n\|_{2,\mu_t}$ converges to $\|G(t,s)|f|\|_{2,\mu_t}$ as $n \to +\infty$,
for every $t\ge s$, and recalling that $|G(t,s)f|\leq G(t,s)|f|$, we get \eqref{formula}, letting $n\to +\infty$.

Finally, by density we can extend \eqref{formula} to any $f\in L^1(\Rd,\mu_s)$ and complete the proof.
\end{proof}

As a consequence of Theorems \ref{ipo_ultrbdd} and \ref{teo_12} we get
the announced ultracontractivity property of $G(t,s)$.

\begin{thm}
Under Hypotheses $\ref{hyp1}$ and $\ref{hyp-2}$ the evolution operator $G(t,s)$ is ultracontractive.
\end{thm}
\begin{proof}
It suffices to prove the claim for $p=1$. For $p>1$ the statement follows from the H\"older inequality.

To conclude the proof, observe that, for every $t>s$,
\begin{equation}
\|G(t,s)\|_{1\to \infty}\leq \|G(t,(t+s)/2)\|_{2\to \infty}\|G((t+s)/2,s)\|_{1 \to 2}.
\label{nascosta}
\end{equation}
\end{proof}

\subsection{Nonautonomous elliptic operators with non-zero potential term}

All the regularizing properties in the previous sections can be extended to nonautonomous operators
with non zero potential term, i.e., operators defined on smooth functions $\zeta$ by
\begin{align*}
({\mathcal A}_c(t)\zeta)(x)=&\sum_{i,j=1}^Nq_{ij}(t)D_{ij}\zeta(x)+\sum_{j=1}^Nb_j(t,x)D_j\zeta(x)-c(t,x)\zeta(x)\\
=&({\mathcal A}(t)\zeta)(x)-c(t,x)\zeta(x),
\end{align*}
for any $t\in I$ and $x\in\Rd$. Besides Hypotheses \ref{hyp1} we assume the following condition.
\begin{hyp}\label{hypc}
$c \in C^{\alpha/2,\alpha}_{\rm loc}(I\times \R^d)$ and $
c_0:=\inf_{I\times\R^d}c>-\infty$.
\end{hyp}

Let $\varphi$, $a$ and $\gamma$ be the function and the constants in Hypothesis \ref{hyp1}(iii). Then,
\begin{eqnarray*}
\mathcal{A}_c(t)\varphi=\mathcal{A}(t)\varphi-c\varphi\le a-(\gamma+c_0)\varphi,\qquad\;\,t\in I.
\end{eqnarray*}
Hence, we can determine a positive constant $\lambda$ such that
${\mathcal A}_c(t)\varphi-\lambda\varphi\le 0$ for any $t\in I$.
We can thus apply the results in \cite{AngLor10Com} which show that a Markov evolution operator $G_c(t,s)$
can be associated to the operator $\mathcal{A}_c(t)$. More precisely, for every $f\in C_b(\Rd)$ and $s \in I$,
$G_c(\cdot,s)f\in C([s,+\infty)\times \Rd)\cap C^{1+\alpha/2,2+\alpha}_{\rm loc}((s,+\infty)\times \R^d)$
is the unique solution of the Cauchy problem
\begin{eqnarray*}
\left\{
\begin{array}{ll}
D_tu(t,x)={\mathcal{A}}_c(t)u(t,x),\quad\quad & (t,x)\in (s,+\infty)\times \Rd,\\[1mm]
u(s,x)= f(x),\quad\quad & x\in \Rd,
\end{array}\right.
\end{eqnarray*}
which satisfies $\|u(t,\cdot)\|_{\infty}\le e^{-c_0(t-s)}\|f\|_{\infty}$.
In the next theorem we will show that $G_c(t,s)$ is
ultracontractive, i.e., it maps $L^p(\Rd, \mu_s)$ into $C_b(\Rd)$ for every $p\ge 1$, where $\{\mu_t: t\in I\}$ the unique tight evolution system of measures for the evolution
operator $G(t,s)$, considered in the previous sections.

\begin{thm}
Assume that Hypotheses $\ref{hyp1}$ and $\ref{hypc}$ hold. If $G(t,s)$ is supercontractive $($resp. ultrabounded, resp. ultracontractive$)$, then
$G_c(t,s)$ is supercontractive $($resp. ultrabounded, resp. ultracontractive$)$ and
\begin{eqnarray*}
\|G_c(t,s)\|_{p\to q}\le C_{p,q}(t-s)e^{-c_0(t-s)},
\end{eqnarray*}
for any $t>s$, $1<p<q<+\infty$ $($resp. $1<p<q=+\infty$, resp. $1\le p<q=+\infty)$, where $C_{p,q}$ is given in
Definition $\ref{definitions}$.
\end{thm}
\begin{proof}
The proof follows immediately observing that a comparison argument based on
\cite[Thm. 2.1]{KunLorLun09Non} shows that
$G_c(t,s)f\leq e^{-c_0(t-s)}G(t,s)f$, for any $t\ge s$ and any nonnegative function $f\in C_b(\Rd)$.
\end{proof}

\begin{rmk}
{\rm
If $c_0\ge 0$, $\{\mu_t: t \in I\}$ is a sub-invariant system of measures for the evolution operator $G_c(t,s)$.
Indeed, since $G_c(t,s)f\leq e^{-c_0(t-s)}G(t,s)f$ for any $t\ge s$ and any nonnegative function $f\in C_b(\Rd)$,
we can estimate
\begin{align*}
\int_{\Rd}G_c(t,s)f\, d\mu_t(x) \leq e^{-c_0(t-s)}\int_{\Rd}G(t,s)f\, d\mu_t(x) \leq \int_{\Rd}f\, d\mu_s(x) ,
\end{align*}
for any $t>s$.}
\end{rmk}

\section{Heat kernel estimates and $L^2$-uniform integrability}
\label{sect-6}
The main goal of this last section is to use regularizing properties of $G(t,s)$ to obtain bounds on the integral kernel
$g_{t,s}$ of $G(t,s)$.
Actually, we show that Hypothesis \ref{ex_1_1}
allows to obtain an $L^{\infty}$-estimate for $g_{t,s}$ and some $L^2$-uniform integrability properties
of $G(t,s)$.

We first prove the following preliminary result.

\begin{lemm}
\label{mdeltalambda}
Assume that Hypotheses $\ref{hyp1}$ and $\ref{ex_1_1}$ hold. Then, for every $\delta, \lambda >0$, there exists a positive constant $\widetilde{M}_{\delta,\lambda}$ such that
\begin{eqnarray*}
(G(t,s)\varphi_\lambda)(x)\leq \widetilde{M}_{\delta,\lambda},
\end{eqnarray*}
for every $t,s\in I$, $t-s\ge \delta$, $x\in \Rd$ and $\lambda>0$. Moreover,
\begin{align*}
\widetilde{M}_{\delta,\lambda}\le&\exp\left [\max\left\{K_0\delta^{\frac{2}{2-\kappa}}\lambda,\left (C_1\lambda^{\frac{\kappa^2}{2(\kappa-2)}}+C_2\lambda^{\frac{\kappa}{2}}\right )^{\frac{2}{k}}\right\}
\right ],\qquad\;\,\delta,\,\lambda>0,
\end{align*}
where $K_0=K_0(\kappa,K_3)$, $C_1=C_1(\Lambda,d,K_3,\kappa)$ and $C_2=C_2(\Lambda,d,K_3,\kappa)$ $($see the proof$)$.
\end{lemm}
\begin{proof}
First of all we point out that, arguing as in the proof of Theorem \ref{ipo_ultrbdd},
 we deduce that assumption \eqref{conv} in Theorem \ref{ultrathm} is satisfied by $h_{\lambda}=g_{\lambda}(y_{0,\lambda})\chi_{[0,y_{0,\lambda}]}+g_{\lambda}\chi_{(y_{0,\lambda},+\infty)}$ 
where $y_{0,\lambda}>1$ denotes the minimum of the function $g_{\lambda}$ defined by
\begin{eqnarray*}
g_{\lambda}(y)=\lambda^{1-\frac{\kappa}{2}}y\left (K_3(\log y)^{\frac{\kappa}{2}}-2\lambda^{\frac{\kappa^2}{2(\kappa-2)}} C_{\kappa}-2\lambda^{\frac{\kappa}{2}}\Lambda d\right ),\qquad\;\,y\ge 1,
\end{eqnarray*}
and $C_{\kappa}$ is any positive constant such that
\begin{eqnarray*}
2\lambda\Lambda y^2\le \frac{K_3}{2}y^{\kappa}+C_{\kappa}\lambda^{\frac{\kappa}{\kappa-2}},\qquad\;\,y\ge 0.
\end{eqnarray*}
Let us observe that
\begin{eqnarray*}
K_3(\log y)^{\frac{\kappa}{2}}-2\lambda^{\frac{\kappa^2}{2(\kappa-2)}} C_{\kappa}-2\lambda^{\frac{\kappa}{2}}\Lambda d\ge
\frac{K_3}{2}(\log y)^{\frac{\kappa}{2}}
\end{eqnarray*}
if and only if
\begin{eqnarray*}
y\ge\exp\left [\left (C_1\lambda^{\frac{\kappa^2}{2(\kappa-2)}}+C_2\lambda^{\frac{\kappa}{2}}\right )^{\frac{2}{k}}\right ]=:P_{\lambda},
\end{eqnarray*}
where $C_1=4C_{\kappa}/K_3$ and $C_2=4\Lambda d/K_3$.
Clearly, if $y\ge P_{\lambda}$ we can estimate
\begin{eqnarray*}
h_{\lambda}(y)\ge \frac{K_3}{2}\lambda^{1-\frac{\kappa}{2}}y(\log y)^{\frac{\kappa}{2}}.
\end{eqnarray*}
Note that $h_{\lambda}(y)>0$ for every $y\ge P_{\lambda}$ and, consequently, $P_{\lambda}>y_{0,\lambda}$.
For any $r\ge 1$, let us set $P_{\lambda,r}=rP_{\lambda}$. Then,
it follows that
\begin{align}
\int_{P_{\lambda,r}}^{+\infty}
\frac{1}{h_{\lambda}(s)}ds
\le & \frac{2}{K_3}\lambda^{\frac{\kappa-2}{2}}\int_{P_{\lambda,r}}^{+\infty}\frac{1}{y(\log y)^{\frac{\kappa}{2}}}dy\notag\\
=&\frac{4}{(\kappa-2)K_3}\lambda^{\frac{\kappa-2}{2}}(\log P_{\lambda,r})^{1-\frac{\kappa}{2}}\notag\\
=&\frac{4}{(\kappa-2)K_3}\lambda^{\frac{\kappa-2}{2}}\left [\left (C_1\lambda^{\frac{\kappa^2}{2(\kappa-2)}}+C_2\lambda^{\frac{\kappa}{2}}\right )^{\frac{2}{k}}+\log r\right ]^{1-\frac{\kappa}{2}}.
\label{est_Plambdar}
\end{align}
Taking into account formula \eqref{est_Plambdar} we deduce that, for any $\delta>0$, the inequality
\begin{eqnarray*}
\frac{2}{K_3}\lambda^{\frac{\kappa-2}{2}}\int_{P_{\lambda,r}}^{+\infty}\frac{1}{y(\log y)^{\frac{\kappa}{2}}}dy\le\delta,
\end{eqnarray*}
is satisfied when
\begin{eqnarray*}
\log r
\ge K_0\delta^{\frac{2}{2-\kappa}}\lambda-\left (C_1\lambda^{\frac{\kappa^2}{2(\kappa-2)}}+C_2\lambda^{\frac{\kappa}{2}}\right )^{\frac{2}{k}},
\end{eqnarray*}
where $K_0=[(\kappa-2)K_3/4]^{2/(2-\kappa)}$.
Hence, if
\begin{eqnarray*}
r=\max\left\{\exp\left[K_0\delta^{\frac{2}{2-\kappa}}\lambda-\left (C_1\lambda^{\frac{\kappa^2}{2(\kappa-2)}}+C_2\lambda^{\frac{\kappa}{2}}\right )^{\frac{2}{k}}\right],1\right\},
\end{eqnarray*}
then
\begin{eqnarray*}
\int_{P_{\lambda,r}}^{+\infty}
\frac{1}{h_{\lambda}(s)}ds
\le\delta.
\end{eqnarray*}
Since $\widetilde{M}_{\delta,\lambda}$ satisfies
\begin{eqnarray*}
\int_{\widetilde{M}_{\delta,\lambda}}^{+\infty}
\frac{1}{h_{\lambda}(s)}ds=\delta,
\end{eqnarray*}
(see \cite[Thm. 4.4]{AngLor10Com}), $\widetilde{M}_{\delta,\lambda}\le P_{\lambda,r}$ and the assertion follows.
\end{proof}

We can now prove the announced heat kernel estimates.

\begin{thm}
Assume that Hypotheses $\ref{hyp1}$ and $\ref{ex_1_1}$ hold. Then, the integral kernel $g_{t,s}$
of $G(t,s)$ satisfies
\begin{eqnarray*}
0 < g_{t,s}(x,y)\le e^{\frac{C}{(t-s)^{\kappa/(\kappa-2)}}},\qquad\;\, I\ni s<t,\;\,0<t-s\le 1,\;\,x,y\in\R^d,
\end{eqnarray*}
where $C$ is a positive constant, depending on $\kappa$, $\eta_0$, $\Lambda$, $d$ and $K_3$.
\end{thm}

\begin{proof}
By Dunford-Pettis theorem (see \cite{DunPet39Lin}), we have $\|G(t,s)\|_{1\to \infty}=\|g_{t,s}\|_{L^{\infty}(\R^{2d})}$.
Formula \eqref{nascosta} implies that $\tilde C_{1,\infty}(t,s)\le C_{1,2}((t-s)/2)C_{2,\infty}((t-s)/2)$.

To estimate $C_{2,\infty}$, we can use \eqref{est_2_pippo} and Lemma
\ref{mdeltalambda} (see also \eqref{mdelta}) which show that
\begin{align*}
C_{2,\infty}((t-s)/2)&=2 e^{\frac{c_1}{\eta_0(t-s)}}M_{\frac{t-s}{4},\frac{1}{\eta_0(t-s)}}
\le 2 e^{\frac{c_1}{\eta_0(t-s)}}\widetilde{M}_{\frac{t-s}{4},\frac{1}{\eta_0(t-s)}}\le e^{\frac{\tilde{C}}{(t-s)^{\kappa/(\kappa-2)}}},
\end{align*}
for any
$0<t-s\le 1$ and some positive constants $c_1$ and $\tilde{C}$, this latter depending on $\kappa, K_3, \eta_0$.
Now, using \eqref{formula}, we get the claim.
\end{proof}

\begin{rmk}
{\rm
We can not expect the
polynomial decay $\|g_{t,s}\|_{L^{\infty}(\R^{2d})}\le C(t-s)^{-\alpha}$ as $t-s\to 0$, for some $\alpha>0$, which is typical of the classical case of
bounded coefficients. Indeed, in the autonomous case, the Varopoulos theorem (see \cite{Var85Har}) implies that a decay of this type occurs if and only if the Sobolev embedding theorems hold,
which, in general, is not the case as the simple example of the standard Gaussian measure in $\R$ shows.}
\end{rmk}

Let us now prove the $L^2$-uniform integrability.

\begin{prop}
Under Hypotheses $\ref{hyp1}$ and $\ref{hyp-2}$, for any $s,t\in I$, with $s<t$, the operator
$G(t,s)$ is $L^2(\Rd,\mu_t)$-uniformly integrable, i.e.,
\begin{eqnarray*}
\lim_{r\to +\infty}\,\sup_{t>s}\sup_{{ f\in L^2(\Rd,\mu_s) }\atop{ \|f\|_{2,\mu_s}\leq 1}}\int_{\{|G(t,s)f| \geq r\}}|G(t,s)f|^2d\mu_t(x) =0.
\end{eqnarray*}

\end{prop}
\begin{proof}
To begin with, let us prove that there exists a positive constant $C$, independent of $f$, such that
\begin{equation}
\int_{A}|G(t,s)f|^2d\mu_t(x)  \leq C\int_{A}\varphi_{\lambda_0}d\mu_t(x) ,
\label{confronto}
\end{equation}
for any $f\in L^2(\Rd,\mu_s)$ with $\|f\|_{2,\mu_s}\le 1$ and any Borel set $A\subset\Rd$, where $\lambda_0=(\eta_0(t-s))^{-1}$.
Note that our assumptions imply that $\varphi_{\lambda_0}\in L^1(\Rd,\mu_t)$ (see \eqref{cond-int-c1-c2}).
We first assume that $f\in C_c(\Rd)$ satisfy $\|f\|_{2,\mu_s}\leq 1$. Integrating \eqref{Harnack} (with $p=2$)
 with respect to $d\mu_t(y)$ and taking \eqref{star} into account, we get
\begin{eqnarray*}
|(G(t,s)f)(x)|^2 \leq 2e^{\frac{R^2}{\eta_0(t-s)}}e^{\frac{|x|^2}{\eta_0(t-s)}}=: C\varphi_{\lambda_0}(x),\qquad\;\, x\in \Rd,
\end{eqnarray*}
where $R$ is any positive constant such that $\mu_t(B(0,R))\ge 1/2$ for any $t\in I$.
From this estimate, \eqref{confronto} follows at once.

Since any function $f\in L^2(\Rd,\mu_s)$, with $\|f\|_{2,\mu_s}\leq 1$, can be approximated by a sequence $(f_n)_n\subset C_c(\Rd)$ satisfying $\|f_n\|_{2,\mu_s}\le 1$ for any $n\in\N$,
estimate \eqref{confronto} can be extended by density to any $f\in L^2(\Rd,\mu_s)$ with $\|f\|_{2,\mu_s}\leq 1$.

Now, recalling that $G(t,s)$ is a contraction from $L^2(\Rd,\mu_s)$ to $L^2(\Rd,\mu_t)$, applying Chebyshev inequality and H\"older inequality, from \eqref{confronto} we easily deduce that
\begin{align*}
\int_{\{|G(t,s)f| \geq r\}}|G(t,s)f|^2d\mu_t(x) \le &
C\int_{\{|G(t,s)f| \geq r\}}\varphi_{\lambda_0}\mu_t(x)\\
\le & C\|\varphi_{2\lambda_0}\|_{1,\mu_t}^{\frac{1}{2}}
(\mu_t(\{|G(t,s)f| \geq r\}))^{\frac{1}{2}}\\
\le & \frac{C}{r}\sup_{t\in I}\|\varphi_{2\lambda_0}\|_{1,\mu_t}^{\frac{1}{2}}.
\end{align*}
The claim now follows at once.
\end{proof}

\paragraph{\bf Acknowledgments.} The authors wish to thank Alessandra Lunardi for helpful comments and for pointing out
\cite{maheux}.

\end{document}